\documentclass[reqno]{amsart}

\usepackage{mathrsfs}

\usepackage{amsmath}
\usepackage{bbm}
\usepackage{mathtools}
\usepackage{amssymb}
\usepackage{graphicx}
\usepackage{color}
\usepackage{latexsym}
\usepackage{comment}
\usepackage{stmaryrd}
\usepackage{MnSymbol}

\usepackage{url}
\usepackage{float}

\graphicspath{{Figures/}}

\newcommand{\p}{permutation}
\newcommand{\ps}{permutations}
\newcommand{\gf}{generating function}
\newcommand{\si}{\sigma}

\newcommand{\tw}{\twoheadleftarrow}
\newcommand{\ts}{\twoheaddownarrow}
\newcommand{\tsw}{\twoheadswarrow}
\newcommand{\tnw}{\twoheadnwarrow}

\newcommand{\ew}{\leftarrow}
\newcommand{\es}{\downarrow}
\newcommand{\esw}{\swarrow}
\newcommand{\enw}{\nwarrow}
\def\emm#1,{{\em #1}}

\newcommand{\cB}{\mathcal B}
\newcommand{\nid}{\noindent}
\DeclareMathOperator{\R}{R}
\DeclareMathOperator{\LB}{LB}
\DeclareMathOperator{\LVB}{LVB}
\DeclareMathOperator{\LHB}{LHB}
\DeclareMathOperator{\B}{B}

\DeclareMathOperator{\UR}{UR}
\DeclareMathOperator{\UHR}{UHR}
\DeclareMathOperator{\UVR}{UVR}
\def\mn{\mbox{-}}

\theoremstyle{plain}
\newtheorem{Theorem}{Theorem}[section]
\newtheorem{Lemma}[Theorem]{Lemma}
\newtheorem{Proposition}[Theorem]{Proposition}
\newtheorem{Corollary}[Theorem]{Corollary}
\newtheorem{Observation}[Theorem]{Observation}
\theoremstyle{definition}
\newtheorem{Definition}[Theorem]{Definition}

\catcode`\@=11
\def\section{\@startsection{section}{1}%
 \z@{.7\linespacing\@plus\linespacing}{.5\linespacing}%
 {\normalfont\bfseries\scshape\centering}}

\def\subsection{\@startsection{subsection}{2}%
  \z@{.5\linespacing\@plus\linespacing}{.5\linespacing}%
  {\normalfont\bfseries\scshape}}

\def\subsubsection{\@startsection{subsubsection}{6}%
 \z@{.5\linespacing\@plus\linespacing}{-.5em}
  {\normalfont\bfseries\itshape}}
\catcode`\@=12

\begin{document}

\author[A. Asinowski]{Andrei Asinowski}

\address{AA: Dept.\ of Mathematics, Technion---Israel Inst.\ of Technology, Haifa 32000, Israel.
Current address: Institut f\"{u}r Informatik, Freie Universit\"{a}t Berlin,
Takustra{\ss}e~9, D-14195 Berlin, Germany}
\email{asinowski@mi.fu-berlin.de} 

\author[G. Barequet]{Gill Barequet}

\address{GB: Dept.\ of Computer Science, Technion---Israel Inst.\ of Technology, Haifa 32000, Israel,
\emph{and}
Dept.\ of Computer Science, Tufts University, Medford, MA 02155}
\email{barequet@cs.technion.ac.il}

\author[M. Bousquet-M\'elou]{Mireille Bousquet-M\'{e}lou}
\address{MBM: CNRS, LaBRI, UMR 5800, Universit\'e de Bordeaux,
351 cours de la Lib\'eration, 33405 Talence Cedex, France}
\email{mireille.bousquet@labri.fr}

\author[T. Mansour]{Toufik Mansour}
\address{TM: Dept.\ of Mathematics, Univ.\ of Haifa, Mount Carmel, Haifa 31905, Israel}
\email{toufik@math.haifa.ac.il}

\author[R. Pinter]{Ron Y.\ Pinter}
\address{RP: Dept.\ of Computer Science, Technion---Israel Inst.\ of Technology, Haifa 32000, Israel}
\email{pinter@cs.technion.ac.il}

\title[Orders induced by segments in floorplans]
{Orders induced by segments in floorplans\\
and $(2 \mn 14 \mn 3, 3 \mn 41 \mn 2)$-avoiding permutations}

\date{Sept. 6th, 2012}


\begin{abstract}
A  floorplan is a tiling of a rectangle by rectangles.
There are natural ways to order the elements---rectangles and segments---of a floorplan.
Ackerman, Barequet and Pinter studied
a pair of orders induced by
neighborhood relations between {rectangles},
and obtained a natural bijection between these pairs
and $(2\mn 41 \mn 3, 3 \mn 14 \mn 2)$-avoiding permutations, also
known as (reduced) Baxter permutations.

In the present paper, we first perform a similar study for a pair of
orders induced by neighborhood relations between \emph{segments}
of a floorplan. We obtain a natural bijection between these
pairs and another family of permutations, namely
$(2\mn 14 \mn 3, 3 \mn 41 \mn 2)$-avoiding permutations.

Then, we investigate relations between
the two kinds of pairs of orders---and, correspondingly, between
$(2\mn 41 \mn 3, 3 \mn 14 \mn 2)$- and $(2\mn 14 \mn 3, 3 \mn 41 \mn
2)$-avoiding permutations. In particular, we prove that the
superposition of both permutations gives a \emm complete, Baxter
permutation (originally called \emm $w$-admissible, by Baxter
and Joichi in the sixties). In other words,  $(2\mn 14 \mn 3, 3 \mn 41 \mn
2)$-avoiding permutations are  the hidden part of complete
Baxter \ps. We enumerate these \ps.
To  our knowledge,  the characterization of these \ps\ in terms of
forbidden patterns and their enumeration are both new results.

Finally, we also study the special case of
the
so-called  \emm guillotine, floorplans.
\end{abstract}

\keywords{Floorplans,
Permutation patterns,
Baxter permutations,
Guillotine floorplans.}

\maketitle

\section{Introduction}

A \emph{floorplan}\footnote{Sometimes called \emm mosaic floorplan, in the literature.}
is a partition of a rectangle
into interior-disjoint rectangles
such that no point belongs to the boundary
of four rectangles  (Fig.~\ref{fig:r_equivalence}).
We call  \emph{segment} of the floorplan any straight line, not
included in the boundary of the partitioned rectangle,  that is the
union of some rectangle sides, and is maximal for this property.
For example, each of the floorplans of Fig.~\ref{fig:r_equivalence}
has four horizontal and four vertical segments.
Since four rectangles of a floorplan never meet, the segments do not
cross, and a meeting of segments
has one of the following forms:
$\leftvdash$, $\upvdash$, $\rightvdash$, $\downvdash$ 
(but not $\bigplus$ ).

\begin{figure}[ht]
\begin{center}
\resizebox{100mm}{!}{\includegraphics{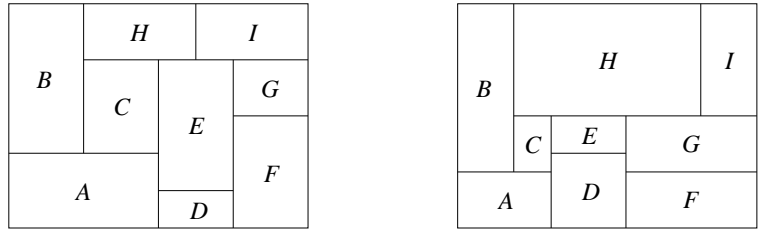}}
\caption{Two R-equivalent floorplans.}
\label{fig:r_equivalence}
\end{center}
\end{figure}

An easy induction shows that the number of segments in
a floorplan is smaller than the number of rectangles by $1$.
Throughout the paper, for a given floorplan $P$,
the number of segments in $P$ is denoted by $n$; accordingly, $n+1$ is
the number of rectangles in $P$.
We say
 that $P$ has \emph{size} $n+1$. For instance,
the floorplans in Fig.~\ref{fig:r_equivalence}
have size $9$.

Many papers have appeared
about floorplans, not only in
combinatorics  but also in the computational
geometry literature~\cite{cls, lprs, ms}. The interest in floorplans
is motivated, in particular, by the fact that
their generation is a critical stage
in integrated circuit layout~\cite{ld, l, rivest, watanabe, wkc},
in architectural designs~\cite{baybars, cousin, flemming, steadman1, steadman2}, etc.

\medskip

The present paper is combinatorial in nature, and describes the
relationship between a pair of natural orders defined on the segments
of a floorplan and certain pattern avoiding permutations. It parallels
a previous paper
of Ackerman, Barequet and Pinter in which a similar
study was carried out for a pair of orders defined on \emm rectangles, of a
floorplan, in connection with the so-called \emm Baxter
\ps,~\cite{abp}. These \ps\ were introduced in the sixties, and
were  originally called  \emm reduced, Baxter \ps, as opposed to \emm
complete, Baxter \ps~\cite{baxter,baxter-joichi,boyce67}. A complete Baxter \p\ $\pi$ has an odd size, say
$2n+1$, and is completely determined by its values  at odd points, $\pi(1), \pi(3), \ldots,
\pi(2n+1)$. After normalization, these values give the associated \emm
reduced, Baxter \p \ $\pi_o$ (the subscript 'o' stands for odd; we
denote by $\pi_e$ the \p\ obtained by normalizing  the list $\pi(2), \pi(4), \ldots,
\pi(2n)$).  Our paper
provides the even part of the theory initiated
in~\cite{abp}: we prove that the
\p\ associated with the segments and  the permutation associated with
the rectangles are respectively the even and odd parts $\pi_e$ and $\pi_o$ of the same
complete Baxter \p\ $\pi$. We also characterize the even parts of Baxter \ps\
in terms of forbidden patterns and enumerate them.

This satisfactory final picture is, however, not our original motivation
for studying segments of floorplans. Instead, our starting point was
the observation that many questions on floorplans deal with segments
rather than rectangles. An interesting example is the
\emph{rectangulation conjecture}~\cite[Conj.~7.1]{abp2}.
It asserts that
for any  floorplan $P$ of size $n+1$,
and any set $\Pi$ of $n$ points of the rectangle such that no two of them lie on the same
horizontal or vertical line,
there exists a floorplan $P'$, R-equivalent\footnote{This equivalence
  relation is defined further down in the introduction.} to $P$,
such that each point of $\Pi$
lies on a segment of $P'$
(see Fig.~\ref{fig:embed} for an illustration).
Such an embedding
of points in segments of a floorplan
is called a \emph{rectangulation}.
The rectangulation conjecture was suggested
in
connection with
the minimization of the total length of segments in certain families
of rectangulations~\cite{lprs}.
Hence, the initial motivation of this paper was to provide us with a
better combinatorial
insight into problems related to floorplans and, in particular, to
rectangulations.

\begin{figure}[ht]
\begin{center}
\resizebox{90mm}{!}{\includegraphics{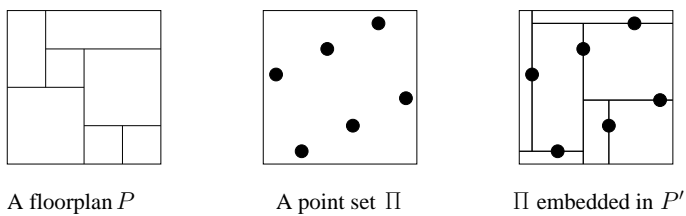}}
\caption{A rectangulation: embedding a point set $\Pi$ in a floorplan $P$.}
\label{fig:embed}
\end{center}
\end{figure}

\medskip
In order to present our results in greater detail, we first need to
describe the related results obtained for rectangles in~\cite{abp}.
In that paper, the authors study a representation
of two
order relations between rectangles in floorplans
by means of permutations.
These order relations are induced by \emph{neighborhood relations},
which are defined as follows.
A rectangle $A$ is a \emph{left-neighbor} of $B$ (equivalently,
$B$ is a \emph{right-neighbor} of $A$) if there is a vertical segment
in the floorplan that contains the right side of $A$ and the left side
of $B$. (Note that
the right side of $A$ and the left side of $B$ may be disjoint.)
Now, the relation \emph{``$A$ is to the left of $B$''} (equivalently,
\emph{``$B$ is to the right of $A$''}), denoted by $A \ew B$, is defined
as the transitive closure of the relation ``$A$ is a left-neighbor of
$B$.'' Finally, the relation $\tw$ is
the reflexive closure of $\ew$.
The terms \emph{$A$ is a below-neighbor of $B$} (equivalently,
\emph{$B$ is an above-neighbor of $A$}) and \emph{$A$ is below $B$}
(equivalently, \emph{$B$ is above $A$}) are defined similarly, as well
as the notation $A \es B$ for ``$A$ is below $B$,''\footnote{Hence, $A \es B$
should be understood as
$\begin{array}{c}
   B \\
   \es \\
   A
\end{array}$; and similarly for $\ts$.}
 and $A \ts B$ for
``$A=B$ or $A \es B$.''
It is easy to see that  the relations $\tw$ and $\ts$ are partial
orders.
In both floorplans of Fig.~\ref{fig:r_equivalence}, we have, among
other relations,
 $A \es I$
(because $A$ is a below-neighbor of $C$, and $C$ is a below-neighbor
of $I$),
$A \ew G$, and $B \ew F$.

The following results
are proved
in~\cite{abp}.
Let $P$ be a floorplan of size $n+1$.
Two distinct rectangles $A$ and $B$ of $P$ are
 in \emph{exactly one}
of the relations $A \ew B$, $B \ew A$, $A \es B$, or $B \es A$.
It follows that the relations $\tsw$ and $\tnw$ between rectangles of
$P$ defined by
 \begin{center}
 \emph{$A \tsw B$ \ \ \ if \ \ \ $A=B$, or $A$ is to the left of $B$,
   or $A$ is below $B$},\\
\emph{$A \tnw B$ \ \ \ if \ \ \ $A=B$, or $A$ is to the left of $B$,
  or $A$ is above $B$},
\end{center}
are \textit{linear} orders
(the signs $\tsw$ and $\tnw$ are intended to resemble the inequality
sign
$\leqslant$).
Each of these orders can be used
to label
 the rectangles of $P$ by $1, 2, \dots, n+1$.
In the $\tsw$~order, the rectangle in the lower left corner
is labeled $1$, and the rectangle in the upper right corner $n+1$.
In the $\tnw$~order, the rectangle in the upper left corner 
is labeled $1$, and the rectangle in the lower right corner $n+1$.
Let $R(P)$ be the sequence
$a_1 \, a_2 \, \dots \, a_{n+1},$ where, for
 all $1 \leq i \leq n+1$, $a_i$ is the label in the $\tsw$~order of
 the rectangle which is labeled $i$ in the $\tnw$~order.
It is clear that $R(P)$ is a permutation of $[n+1]:= \{1, 2, \dots, n+1\}$;
we call it the \emph{R-permutation} of $P$.
Loosely speaking, $R(P)$ is
obtained by labeling the rectangles
according to the $\tsw$~order, and then reading these labels while
passing the rectangles according to the $\tnw$~order.
Fig.~\ref{fig:r_ordering} shows the construction of the R-permutation of a floorplan $P$;
the right part of the figure is the graph of $\rho=R(P)$, that is,
the point set $\{(i, \rho(i)): \ 1 \leq i \leq n+1 \}$.

\begin{figure}[ht]
\begin{center}
\resizebox{120mm}{!}{\includegraphics{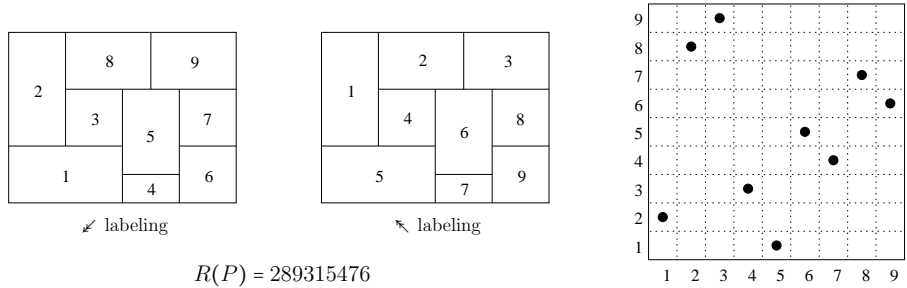}}
\caption{Constructing the R-permutation of a floorplan $P$.}
\label{fig:r_ordering}
\end{center}
\end{figure}

Two floorplans $P_1$ and $P_2$ of size $n+1$
are said to be \emph{R-equivalent}\footnote{In~\cite{abp}, two
  R-equivalent floorplans are actually treated as
two representations of the same floorplan.}
if there exists a labeling of the rectangles of $P_1$
by $A_1, A_2, \ldots, A_{n+1}$
and a labeling  of the rectangles of $P_2$
by $B_1, B_2, \ldots, B_{n+1}$ such that for all $k, m \in
[n+1]$,
the rectangles $A_k$ and $A_m$ exhibit the same
neighborhood relation as $B_k$ and $B_m$.
The two floorplans in Fig.~\ref{fig:r_equivalence} are R-equivalent:
this follows from the labeling presented in this figure (in fact, $A, B, C, \ldots I$
is the $\tsw$-order).
It is easy to prove that two floorplans are R-equivalent if
and only if they have the same R-permutation.

The main results of~\cite{abp} state that
for any floorplan $P$, its R-permutation is a $(2\mn 41 \mn 3, 3 \mn
14 \mn 2)$-avoiding permutation\footnote{This notation is explained
in Section~\ref{sec:patterns}.},
originally called  \emph{(reduced) Baxter permutation};
moreover, $R$ is a bijection between R-equivalence classes of floorplans and
$(2\mn 41 \mn 3, 3 \mn 14 \mn 2)$-avoiding permutations.
Through this correspondence, the size of a floorplan becomes the
size of the \p, and
the order relations between rectangles in $P$ can be easily read from
$R(P)$.

\medskip

We can now describe our results   in greater detail.

In the first part of the paper, we develop for \emm
segments, a theory that parallels
the theory developed for \emm rectangles, in~\cite{abp}. We define two
order relations between segments (Section~\ref{sec:orders}), which leads to the notion of S-equivalent floorplans\footnote{In the
notion of R-equivalence and S-equivalence, \emph{R} stands for
\emph{rectangles} and   \emph{S} for \emph{segments}.}.
Then we use these orders
to construct a permutation $S(P)$
called
the \emph{S-permutation} of $P$. 
In Section~\ref{sec:bij} we prove that
S-permutations coincide with $(2\mn14\mn3, 3\mn41\mn2)$-avoiding
permutations, and that $S$, regarded as a function from S-equivalence
classes 
 to $(2\mn14\mn3, 3\mn41\mn2)$-avoiding
permutations, is a bijection. The description of $S$ and $S^{-1}$ are fairly simple (both can be constructed in linear time), but as often, the
proof remains technical, despite our efforts to write it carefully.

In the second part of the paper, we
super-impose our theory with the analogous theory developed for
rectangles in~\cite{abp}.
In Section~\ref{sec:seg_rec} we
show that R-equivalence of floorplans implies their S-equivalence
(this means that the R-equivalence
refines
the S-equivalence),
and explain how $S(P)$ can be constructed directly from $R(P)$.
This construction shows that $S(P)$ and $R(P)$, combined together,
form the so-called \emm complete Baxter \p, associated with $R(P)$, as
defined in the seminal papers on Baxter \ps~\cite{baxter,baxter-joichi,boyce67}.
We also describe in terms of $R$ when two floorplans give rise to the
same S-permutation.
This is another difficult proof, but we need  this description to
express the number of $(2\mn14\mn3, 3\mn41\mn2)$-avoiding permutations
in terms of the number of Baxter \ps\ (Section~\ref{sec:enum}).

To finish, in Section~\ref{sec:guil} we characterize and enumerate
S-permutations corresponding to the so-called \emm guillotine, floorplans;
a similar study was carried out in~\cite{abp} for R-permutations.
We end in Section~\ref{sec:summary} with a few remarks.

\section{Orders between segments of a floorplan}
\label{sec:orders}

In this section we define
neighborhood relations between {segments} of a floorplan,
 use them to define two partial orders (denoted $\tw$ and $\ts$) and two linear orders (denoted $\tsw$ and $\tnw$),
and prove several facts about these orders.  Most of them are
analogous to facts about the orders on rectangles mentioned
in the introduction, and proved in~\cite{abp}.

\begin{Definition}
Let $P$ be a floorplan.
Let $I$ and $J$ be two segments in $P$.
We say that $I$ is a \emph{left-neighbor} of $J$
(equivalently, $J$ is a \emph{right-neighbor} of $I$) if one of the following holds:
\begin{itemize}
  \item $I$ and $J$ are vertical, and there is \emph{exactly one} rectangle $A$ in $P$
    such that the left side of $A$ is contained in $I$ and the right
    side of $A$ is contained in $J$;
  \item $I$ is vertical, $J$ is horizontal, and the left endpoint of $J$ lies in $I$; or
  \item $I$ is horizontal, $J$ is vertical, and the right endpoint of
    $I$ lies in $J$.
\end{itemize}
The terms ``\emph{$I$ is a below-neighbor of $J$}'' (equivalently,
``\emph{$J$ is an above-neighbor of $I$}'')
are defined similarly.
\end{Definition}

Typical examples are shown in Fig.~\ref{fig:left}.

\begin{figure}[h]
\begin{center}
\resizebox{130mm}{!}{\includegraphics{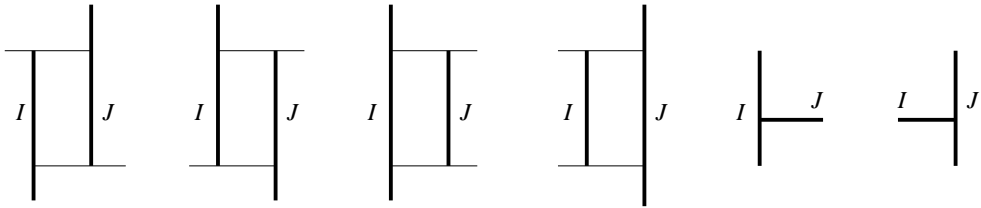}}
\caption{The segment $I$ is a left-neighbor of the segment $J$.}
\label{fig:left}
\end{center}
\end{figure}

Note that a horizontal segment $I$ has at most one left-neighbor and
at most one right-neighbor
(no such neighbor(s) when the corresponding endpoint(s) of $I$ lie on the boundary),
which are both vertical segments.
In contrast, a vertical segment
may have several left- and right-neighbors, which
may be horizontal or vertical.
This is illustrated in Fig.~\ref{fig:right-neighbors}.

\begin{figure}[h]
\begin{center}
\resizebox{35mm}{!}
{\includegraphics{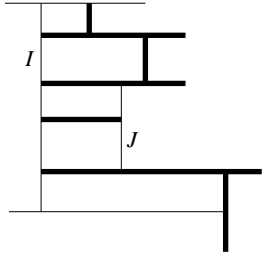}}
\caption{The right-neighbors of a vertical segment $I$ (thick segments).
Note that the vertical segment
$J$ is \emm not, a right-neighbor of $I$.}
\label{fig:right-neighbors}
\end{center}
\end{figure}

\begin{Definition}
 The relation \emph{``$I$ is to the left of $J$''}
(equivalently, \emph{$J$ is to the right of $I$}), denoted by $I \ew J$, is
the transitive closure of the relation ``$I$ is a left-neighbor of $J$.''
The relation $\tw$
is the reflexive closure of $\ew$.
  The relations $I \es J$ (``\textit{$I$ is below $J$}'') and $I \ts J$
  (for ``$I=J$ or $I \es J$'') are defined similarly.
\end{Definition}

\begin{Observation} \label{the:par}
The relations $\tw$ and $\ts$ are partial orders.
\end{Observation}

\begin{proof}
We prove the claim for the relation $\tw$.
Reflexivity and transitivity follow
 from the definition.
For antisymmetry, note that $I \ew J$ and $J \ew I$ cannot hold
simultaneously because if $I \ew J$, then any interior point of $I$
has a smaller abscissa than any interior point of $J$.
\end{proof}

The following observation may help to understand the $\tw$
order. If $I$ and $J$ are vertical segments and right-left neighbors,
let us create a horizontal edge, called \emph{traversing edge}, in
the rectangle~$A$ that lies between them. Fig.~\ref{fig:traversing-edges} shows
a chain of neighbors in the $\tw$ order, starting from a vertical
segment $I$, and the corresponding traversing edges (dashed lines).

\begin{figure}[h]
\begin{center}
\resizebox{70mm}{!}
{\includegraphics{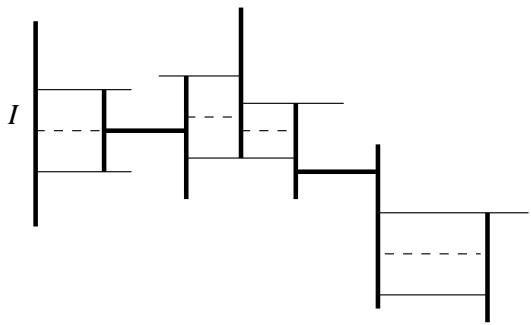}}
\caption{A chain in the $\tw$ order (thick segments), and the
  corresponding traversing edges (dashed lines).}
\label{fig:traversing-edges}
\end{center}
\end{figure}

\begin{Observation}
\label{the:path}
Assume $I\tw J$. Then any point of $J$ lies weakly to the right of any
point of $I$ (that is, its abscissa is at least as large).

Let $x$ (respectively, $y$) be a point of minimal
(respectively, maximal) abscissa on $I$ (respectively, $J$). Then
there exists a polygonal line 
from $x$ to $y$ formed of vertical and horizontal sections, such that
\begin{enumerate}
\item [--] each vertical section is part of a vertical segment of the
  floorplan $P$,
\item [--] each horizontal section is an (entire) horizontal segment
  of $P$, or a traversing edge of $P$, visited from left to right,
\item [--] if $I$ (respectively, $J$) is horizontal, it is entirely included
  in the polygonal line .
\end{enumerate}
\end{Observation}
It suffices to prove these properties when $J$ is a right-neighbor of
$I$, and  they are obvious in this case 
(see Fig.~\ref{fig:left}).

\begin{Lemma}
\label{the:covering}
In the $\tw$ order, $J$ covers $I$ if and only if $J$ is a
right-neighbor of $I$. A similar statement holds for the $\ts$ order.
\end{Lemma}

\begin{proof}
Since $\tw$ is constructed as the transitive closure of the
left-right neighborhood relation, every covering relation is a
neighborhood relation.

Conversely, let us prove that any neighborhood
relation is a covering relation. Equivalently, this means that the
right-neighbors of any segment $I$ form an antichain. If $I$ is
horizontal, it has at
most one right-neighbor, and there is nothing to prove. Assume $I$ is
vertical (as in Fig.~\ref{fig:right-neighbors}), and that two of its
neighbors, $J_1$ and $J_2$, satisfy $J_1\ew J_2$.
By the first part of Observation~\ref{the:path}, $J_2$ cannot be horizontal (its
leftmost point would then lie on $I$, leaving no place for $J_1$). 
Hence $J_2$ is vertical. The  possible configurations of $I$ and $J_2$
are depicted in the first four cases of Fig.~\ref{fig:left}. Let $x$
(resp. $y$) be a point of $I$ (resp. $J_2$). By
Observation~\ref{the:path}, there exists a
polygonal line  from $x$ to $y$ that visits a point of $J_1$. This
rules out the third and fourth cases of Fig.~\ref{fig:left} (the line
would be reduced to a single traversing edge). By symmetry we can
assume that $I$ and $J_2$ are as in the first case of Fig.~\ref{fig:left}. Then the polygonal
line, which is not a single traversing edge, has to leave $I$ at a point that lies lower than the lowest point
of $J_2$, and to reach $J_2$ at a point that lies higher than the
highest point of $I$: this means that it crosses two horizontal
segments, which is impossible given the description of this line.
\end{proof}

\begin{Lemma}\label{com}
\label{the:seg}
Let $I$ and $J$ be two different segments in a floorplan
$P$. Then exactly one of the relations: $I \ew J$, $J \ew I$,  $I \es
J$, or $J \es I$, holds.
\end{Lemma}
\begin{proof}
Assume without loss of generality that $I$ is a horizontal segment.
Construct the \emph{NE-sequence} $K_1, K_2, \ldots$ of $I$ as follows
(see Fig.~\ref{fig:compat} for an illustration):
$K_1$ is the right-neighbor of $I$,
$K_2$ the above-neighbor of
$K_1$, $K_3$ the right-neighbor of $K_2$,
and so on, until the boundary is reached.
Construct similarly the SE-, NW-, and SW-sequences of $I$.
These sequences partition the rectangle into four regions
(or 
fewer, if some endpoints of $I$ lie on the boundary); each
segment of $P$ (except $I$ and those belonging to either of the
sequences) lies in 
exactly one of them.
Also, if $J$ is in the interior of a region, then its neighbors are
either in the same region, or in one of the sequences that bound the region.

It is not hard to see that the vertical segments of the NE-sequence are to the
{right} of $I$, while horizontal segments are {above} $I$.
A horizontal segment  $K_{2i}$ cannot be to
the left of $I$, since it ends to the right of $I$. Let us prove that
$K_{2i}$ cannot be
the right of $I$ either. Assume this is the case, and consider the
polygonal line  going from the leftmost point
of $I$ to the rightmost point of $K_{2i}$, as described in
Observation~\ref{the:path}. The last section of this  line  is
$K_{2i}$. Hence the  line  has points in the interior of the region
comprised between the NW- and NE-sequences.
But since the  line 
always goes to the right,
and follows entirely every horizontal
segment it visits, it can never enter the interior of this region. Thus
$K_{2i}$ cannot be to the right of $I$.
Thus, its only relation to $I$ is $I \es K_{2i}$.
Similar  arguments apply for vertical segments of the NE-sequence, and for the other three sequences.

Consider now a segment $J$ that lies, for instance, in the North region
(that is
 the region bounded by the NE-sequence, the NW-sequence, and the
 boundary;
the case of other regions is similar).
Then $I$ is below $J$: if
we consider the below-neighbors of $J$, then their below-neighbors, and so
on, then we necessarily reach one of the horizontal segments of the
NW- or NE-sequence,
which, as we have seen, are above $I$ (we cannot reach a vertical
segment of the sequences without 
reaching a horizontal segment first).

Hence, we have that $I \es J$; it remains to prove that
the
other three relations are impossible.

First, $J \es I$ is impossible since the relation $\es$ is antisymmetric.
To prove that $J$ cannot be to the right of $I$, we argue as we did
for $K_{2i}$: the polygonal line from $I$ to $J$ starting from the leftmost
point of $I$ cannot enter the North region. 
Symmetrically, $J$ cannot
be to the left of $I$. This completes the proof.
\end{proof}

\begin{figure}[ht]
\begin{center}
\resizebox{50mm}{!}{\includegraphics{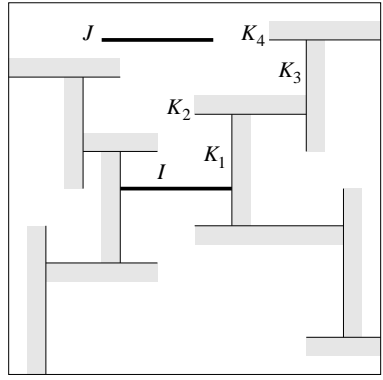}}
\caption{Four regions determining the relationship between $I$ and
  other segments of the floorplan.}
\label{fig:compat}
\end{center}
\end{figure}

\begin{Definition}
 The relations $\tsw$ and $\tnw$ between segments of a floorplan
 $P$ are defined by: 
\begin{center}
\emph{$I \tsw J$ \ \ \ if \ \ \ $I=J$, or $I$ is to the left of $J$,
  or $I$ is below $J$},\\
\emph{$I \tnw J$ \ \ \ if \ \ \ $I=J$, or $I$ is to the left of $J$,
  or $I$ is above $J$}.
\end{center}
We also write $I \esw J$ when $I \tsw J$ and $I \neq J$; and $I \enw
J$ when $I \tnw J$ and $I \neq J$.
\end{Definition}

\begin{Proposition}
\label{the:lin}
The relations $\tsw$ and $\tnw$ are linear orders.
\end{Proposition}

\begin{proof}
We prove the claim for the relation $\tsw$.
Reflexivity follows 
 from the definition.
Antisymmetry follows from
the fact that $\tw$ and $\ts$ are order relations,
and from Lemma~\ref{the:seg}.

For transitivity, assume that $I \esw J$ and $J \esw K$. If $I \ew J$
and $J \ew K$ (respectively, $I \es J$ and $J \es K$), then we have $I
\ew K$ (respectively, $I \es K$)
by the transitivity of $\ew$ (respectively, $\es$).
Assume now that $I \ew J$ and $J \es K$ (the case $I \es J$ and $J \ew
K$ is proven similarly). By Lemma~\ref{the:seg}, $I=K$ is impossible,
and we have either $I \ew K$, $K \ew I$, $I \es K$, or $K \es
I$. However, the combination of $K \ew I$ and $I \ew J$ yields $K \ew
J$, contradicting  the assumption that $J \es
K$ (by Lemma~\ref{com}). Similarly, combining $K \es I$ with $J \es K$ yields $J \es I$,
contradicting the assumption that $I \ew J$. Therefore, we have either
$I \ew K$ or $I \es K$; 
in particular, $I \esw K$.

Linearity follows from Lemma~\ref{the:seg}.
\end{proof}

\begin{Observation}
  The orders $\tw$ and $\ts$ can be recovered from $\tsw$ and $\tnw$.
  Indeed, $I\tw J$ if and only if  $I\tsw J$  and $I\tnw J$; moreover,
  $I\ts J$ if and only if  $I\tsw J$  and   $J\tnw I$.
\end{Observation}

\emph{Throughout the paper, the $i$th segment
in the  $\tnw$~order ($1 \leq i \leq n$) will be denoted by $I_i$.}
See Fig.~\ref{fig:s_equivalence} for examples.

We now explain how to determine $I_{i+1}$ among the neighbors of $I_i$.
By Lemma~\ref{the:covering}, 
$I_{i+1}$ is either a right- or below-neighbor of $I_i$.
There are several cases depending on the existence of 
these neighbors
and the relations between them. 
For a horizontal segment $I$,
we denote by $\R(I)$ the right-neighbor of $I$ (when it exists).
By Lemma~\ref{the:covering},
the below-neighbors of $I$ form an antichain of the $\ts$ order.
Since $\tnw$ is a linear order, they are
totally ordered for the $\tw$ order.
By the first part of Observation~\ref{the:path},
the leftmost is also the smallest, denoted $\LB(I)$.
Thus $\LB(I)$ is either
$\LVB(I)$ (the leftmost \emph{vertical} below-neighbor of $I$)
or
$\LHB(I)$ (the leftmost \emph{horizontal} below-neighbor of $I$).
Similarly, for a vertical segment $I$, we denote
by $\B(I)$ the below-neighbor of $I$;
by $\UR(I)$ the highest\footnote{The letter U stands for \emm up,.}
 right-neighbor of $I$,
and by $\UHR(I)$ (respectively, $\UVR(I)$)
the highest horizontal (respectively, vertical)
right-neighbor of $I$.
 Fig.~\ref{fig:next_se}  illustrates the following observation
(it is assumed that
all 
candidates for
$I_{i+1}$ are 
 depicted.
The dashed lines belong to the boundary).

\begin{Observation}\label{se_nei}
\label{th:se_nei}
Let $I_i$ be a segment in a floorplan $P$
of size $n+1$.
If $I_i$ is horizontal, then $I_{i+1}$ is either $\R(I_i)$ or $\LB(I_i)$.
More precisely,
\begin{enumerate}
\item If none of  $\R(I_i)$ and $\LB(I_i)$ exists, then
     $I_i$ is the last segment in the $\tnw$~order (that is, $i=n$).
\item If exactly one of
     $\R(I_i)$  and $\LB(I_i)$ exists, then $I_{i+1}$ is this segment.
\item
If $\LVB(I_i)$  exists,
then $I_{i+1}=\LB(I_i)$. This segment is
  $\LHB(I_i)$ if it exists, and otherwise $\LVB(I_i)$.
  \item If $\LVB(I_i)$ does not exist but $\LHB(I_i)$ and $\R(I_i)$ exist, then
  \begin{itemize}
    \item If the join of $\LHB(I_i)$ and $\R(I_i)$ is of type $\leftvdash$, then $I_{i+1}=\LHB(I_i)$.
    \item If the join of $\LHB(I_i)$ and $\R(I_i)$ is of type $\upvdash$, then $I_{i+1}=\R(I_i)$.
  \end{itemize}
 \end{enumerate}

If $I_i$ is vertical, then $I_{i+1}$ is either $\B(I_i)$, $\UHR(I_i)$, or $\UVR(I_i)$ (the details are similar to those in the case of a horizontal segment).
\end{Observation}

\begin{figure}[h]
\begin{center}
\resizebox{165mm}{!}{\includegraphics{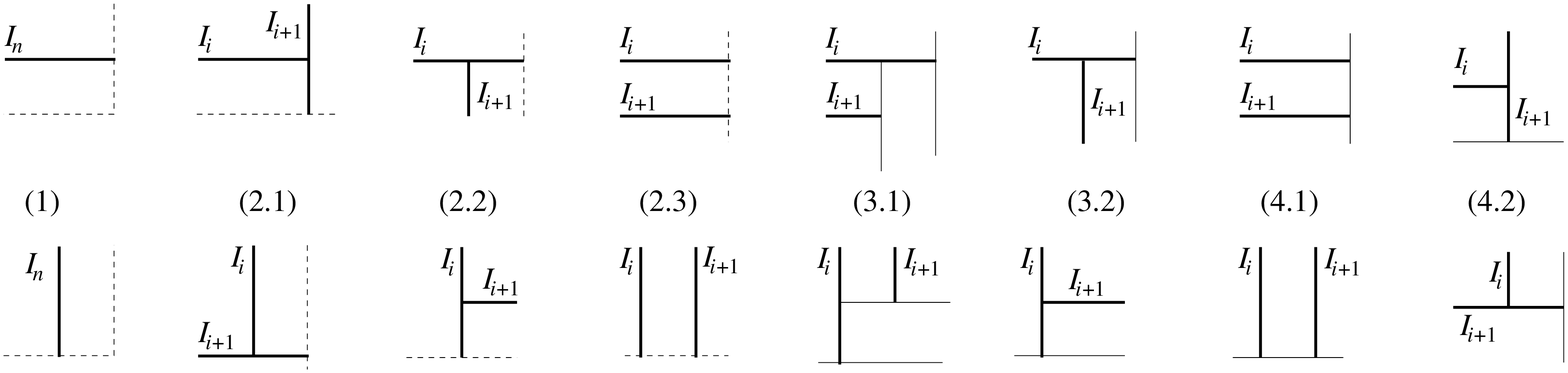}}
\caption{The segment $I_{i+1}$ follows
$I_i$ in  the $\tnw$~order
(Top: $I_i$ is horizontal. Bottom: $I_i$ is vertical).}
\label{fig:next_se}
\end{center}
\end{figure}

One can in fact construct
simultaneously, and in a single pass, the labeling of rectangles and
segments.

\begin{Proposition}
\label{the:r2s_1}
Let $P$ be a floorplan of size $n+1$, and let $A_k$ denote the rectangle labeled $k$ in the $\tnw$~order.
For 
$1 \leq k \leq n$, the following property, illustrated in Fig.~\ref{fig:r2s_1},
holds:
\begin{itemize}
  \item If the segments forming the SE-corner of $A_k$ have a $\upvdash$ join,
  let $J_k$ be the segment containing the right side of $A_k$.
  Then $A_{k+1}$ is the highest
rectangle whose left side is contained in $J_k$.
  \item If the segments forming the SE-corner of $A_k$ have a $\leftvdash$ join,
  let $J_k$ be the segment  containing the lower side of $A_k$.
  Then $A_{k+1}$ is the leftmost 
rectangle whose upper side is contained in $J_k$.
\end{itemize}
In both cases,  $J_k$ is the $k$\emph{th} segment in the $\tnw$~order
of segments, denoted so far by $I_k$.
\end{Proposition}
\begin{proof}
By definition of the $\tnw$~order,
$A_{k+1}$ is either a right-neighbor or a below-neighbor of $A_k$.
If there is a $\upvdash$ join in the SE-corner of $A_k$, then all the right-neighbors of $A_k$ are above all its below-neighbors. Therefore, $A_{k+1}$ is the topmost among them.
If there is a $\leftvdash$ join in the SE-corner of $A_k$, then all
the below-neighbors of $A_k$ are to the left of all its
right-neighbors. Therefore, $A_{k+1}$ is the leftmost among them.

\begin{figure}[ht]
\begin{center}
\resizebox{90mm}{!}{\includegraphics{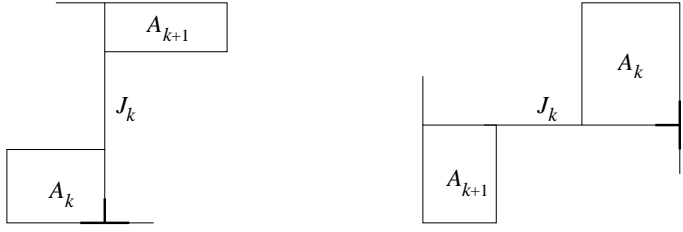}}
\caption{The rectangle $A_{k+1}$ follows
$A_k$ in the $\tnw$~order.}
\label{fig:r2s_1}
\end{center}
\end{figure}

To prove the second statement, we observe it directly for $k=1$, and
proceed by induction.
One has to examine several cases, depending on whether the segments in the SE-corners of $A_k$ and of $A_{k+1}$  have $\upvdash$ or $\leftvdash$
joins.
In all cases, $J_{k+1}$ is found to be
the immediate successor of $J_k$ in the
$\tnw$~order, as described in Observation~\ref{th:se_nei}.
See Fig.~\ref{fig:r2s_2} for several typical situations.
\end{proof}

\begin{figure}[ht]
\begin{center}
\resizebox{130mm}{!}{\includegraphics{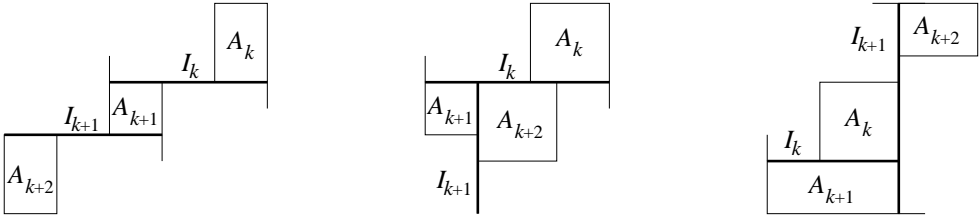}}
\caption{Successors of segments and rectangles for the $\tnw$~orders.}
\label{fig:r2s_2}
\end{center}
\end{figure}

The group of symmetries of the square acts on floorplans
(when floorplans are drawn in a square).
It is thus worth
examining how the orders are transformed when applying such
symmetries. As this symmetry group is
generated by two generators, for instance the reflections
in the
first diagonal and across a horizontal line, it suffices to study
these two transformations. The following proposition easily follows
from the description of the neighborhood relations of Fig.~\ref{fig:left}.

\begin{Proposition}
\label{obs:sym-orders}
Let $P$ be a (square) floorplan, and let $P'$ be obtained by reflecting
$P$ 
in the first diagonal. If $I$ is a segment of $P$, let $I'$
denote the corresponding segment of $P'$.
Then
$$
\begin{array}{lll}
I\tw J &\Leftrightarrow& I'\ts J', \\
I\ts J &\Leftrightarrow &I'\tw J', \\
I\tsw J &\Leftrightarrow& I'\tsw J' , \\
I\tnw J &\Leftrightarrow& J'\tnw I'. \\
\end{array}
$$
If instead $P'$ is obtained by reflecting $P$
in a horizontal line,
$$
\begin{array}{lll}
I\tw J &\Leftrightarrow& I'\tw J', \\
I\ts J &\Leftrightarrow &J'\ts I', \\
I\tsw J &\Leftrightarrow& I'\tnw J', \\
I\tnw J &\Leftrightarrow& I'\tsw J'. \\
\end{array}
$$
\end{Proposition}
One consequence of this proposition is that  a
half-turn rotation of $P$
reverses all four orders.
We shall also use the fact that,  if $P'$ is obtained by applying a
clockwise quarter-turn rotation to $P$, then $I\tnw J \Leftrightarrow
J'\tsw I'$.

\section{A bijection between S-equivalence classes of floorplans
and \\ \boldmath{$(2\mn14\mn3, 3\mn41\mn2)$}-avoiding permutations}
\label{sec:bij}
In this section we define S-equivalence of floorplans
and construct
a map $S$ from floorplans to
permutations.
We show that $S$ 
induces an injection from S-equivalence classes 
 to permutations. We then 
characterize the class of
permutations obtained from floorplans 
in terms of (generalized) patterns.

\subsection{S-equivalence}
\label{sec:s_eq}

\begin{Definition}
Two floorplans $P_1$ and $P_2$ of size $n+1$ are {\emph{S-equivalent}}
if it is possible to label the segments of $P_1$ by $I_1,
I_2, \dots, I_n$ and the segments of $P_2$ by $J_1, J_2, \dots, J_n$
so that for all $k, m \in [n]$, the segments $I_k$ and $I_m$
exhibit 
 the same neighborhood relation as $J_k$ and $J_m$.
\end{Definition}

Fig.~\ref{fig:s_equivalence} shows two S-equivalent floorplans:
in both cases, the left-right \emph{neighborhood} relations are
$1 \ew 4, \ 2 \ew 4, \ 3 \ew 4, \ 4 \ew 5, \ 4 \ew 6$, and the below-above \emph{neighborhood}
relations are $2 \es 1, \ 3 \es 2, \ 6 \es 5$.
These floorplans are not R-equivalent, as can be seen by constructing
their R-permutations.
We will prove  in Section~\ref{sec:seg_rec} that, conversely, R-equivalence
implies S-equivalence.

\begin{figure}[h]
\begin{center}
\resizebox{100mm}{!}{\includegraphics{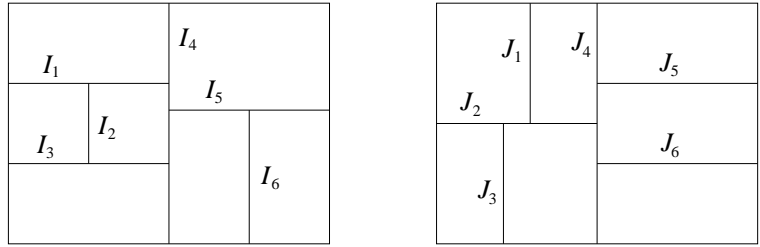}}
\caption{Two S-equivalent (but not R-equivalent)
floorplans.}
\label{fig:s_equivalence}
\end{center}
\end{figure}

\subsection{S-permutations}
\label{sec:perm}

Let $P$ be a floorplan of size $n+1$.
There are $n$ segments in $P$.
Let $S(P)$ be the sequence $b_1, b_2, \dots, b_n$, where $b_i$ is the
label in the $\tsw$~order of the segment labeled $i$ in the $\tnw$~order,
for all $1 \leq i \leq n$.
Then $S(P)$ is a permutation of $[n]$; 
we call it the \emph{S-permutation} of $P$ and denote it by $S(P)$.
Equivalently, if $I_1, \ldots, I_n$ is the list of segments in the
$\tnw$ order, then $I_{\sigma^{-1}(1)}, \ldots, I_{\sigma^{-1}(n)}$ is
the list of segments in the $\tsw$ order,
with $\si=S(P)$.
Since the $\tnw$- and $\tsw$-orders on segments can be determined in
linear time (Proposition~\ref{the:r2s_1}), the S-permutation is also
constructed in linear time.
An example is shown in Fig.~\ref{fig:ex1}.

Thus, we assign a permutation to a floorplan in a way
similar to that used in~\cite{abp}, but this time we use
order
 relations between segments rather than rectangles.
Note that $S(P)$ is a permutation of $[n]$, while $R(P)$ is a permutation of $[n+1]$.

By  definition of $S(P)$, if a segment of $P$ is labeled $i$ in the
$\tnw$~order and $j$ in the $\tsw$~order, then $S(P)(i)=j$. In other
words, the graph of $S(P)$ contains the point $(i,j)$, which will be
denoted by $N_i$.

\begin{figure}[ht]
\begin{center}
\resizebox{120mm}{!}{\includegraphics{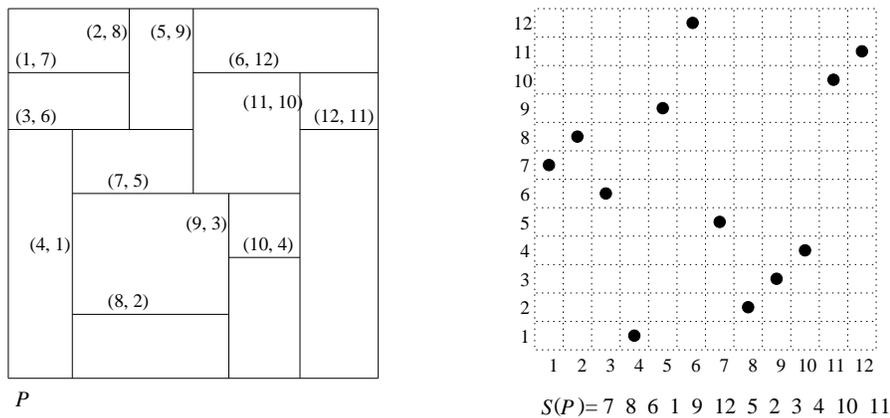}}
\caption{A floorplan $P$, with segments labeled $(i, j)$, where $i$ (respectively, $j$) is the label according to the $\tnw$ (respectively, $\tsw$) order,
and the corresponding S-permutation.}
\label{fig:ex1}
\end{center}
\end{figure}

It follows from Proposition~\ref{obs:sym-orders} that the map $S$ has
a simple behavior with respect to symmetries.
\begin{Proposition}
\label{obs:sym-S}
  Let $P$ be a (square) floorplan, and $P'$ be obtained by reflecting
$P$
in the first diagonal. Let $\si=S(P)$ and  $\si'=S(P')$. Then
  $\si'$ is obtained by reading $\si$ from right to left or
  equivalently, by reflecting the graph of $\si$ 
in a vertical  line.

If instead $P'$ is obtained by reflecting $P$ 
in a horizontal
line, then $\si'=\si^{-1}$. Equivalently, $\si'$
 is obtained  by reflecting the graph of $\si$ 
in the first diagonal.
\end{Proposition}
\begin{proof}
The following statements are equivalent:
\begin{itemize}
\item $\si(i)=j$,
\item there exists a segment of $P$ that has label $i$ in the
  $\tnw$-order and $j$ in the $\tsw$-order (by definition of $S$),
\item there exists a segment of $P'$ that has label $n+1-i$ in the
  $\tnw$-order and $j$ in the $\tsw$-order (by Proposition~\ref{obs:sym-orders}),
\item $\si'(n+1-i)=j$.
\end{itemize}
This proves the first result. The proof of the second result is similar.
\end{proof}
Since the two reflections of Proposition~\ref{obs:sym-S} generate the group of
symmetries of the square, we can describe what happens for the other elements of
this group: applying a rotation to $P$ boils down to applying the
same rotation to $S(P)$, and
reflecting $P$ 
in $\Delta$, a symmetry axis of the bounding square,
boils down to
reflecting $S(P)$ 
in $\Delta'$, a line
obtained by rotating $\Delta$ of  $45^\circ$ in counterclockwise direction.
These properties will be extremely useful to decrease the number of
cases we have to study in certain proofs.


\medskip

We will now prove that $S(P)$ characterizes the S-equivalence class of
$P$. Clearly, two S-equivalent floorplans give rise to the same
orders, and thus to the same S-permutation.
Conversely, let us define neighborhood relations between points in the graph of a
permutation $\si$ as follows.
Let $N_{i} = (i, \sigma(i))$, $N_{j} = (j, \sigma(j))$ be two points
in the graph of $\sigma$.
If $i<j$ and $\sigma(i)<\sigma(j)$, then the point $N_{j}$ is  \emph{to
  the NE} of the point $N_{i}$.
If, in addition, there is no $i'$ such that $i<i'<j$ and
$\sigma(i)<\sigma(i')<\sigma(j)$, then 
$N_{j}$ is a  \emph{NE-neighbor} of 
$N_{i}$.
In a similar way we define when  $N_{j}$ is \emph{to the SE /
  SW / NW} of  $N_{i}$,
and when the point $N_{j}$ is
 a \emph{SE- / SW- / NW-neighbor} of  $N_{i}$.
For example, in the graph 
of Fig.~\ref{fig:ex1}, the points
$(1,7)$, $(2, 8)$, $(3,6)$, $(5,9)$ and $(6,12)$ are to the NW of
$N_7=(7,5)$; among them, $(3,6)$, $(5,9)$ and $(6,12)$ are
NW-neighbors of $N_7$.

The neighborhood relations between segments of $P$ correspond to the
neighborhood relations in the graph of $S(P)$ in the following way.

\begin{Observation}
\label{the:nei_dia}
Let $P$ be a floorplan, and let $I_{i}$ and $I_{j}$ be two
segments in $P$.

The segment $I_{j}$ is to the right of $I_i$ if and only if the point
$N_{j}$ lies to the  NE of $N_i$. Consequently, $I_{j}$ is a
right-neighbor  of $I_i$ if and only if $N_{j}$ is a NE-neighbor of $N_i$.

Similar statements hold for the other directions: left (respectively, above,
below) neighbors in segments correspond to SW- (respectively, NW-,
SE-) neighbors in points.
\end{Observation}

\begin{proof}
The segment $I_{j}$ is to the right of $I_i$ if and only if $I_i\tsw
I_{j}$ and $I_i\tnw I_{j}$. By construction of $\si=S(P)$, this means
that $i<j$ and $\si(i) <\si(j)$. Equivalently, $N_{j}$ lies to the  NE
of $N_i$.
\end{proof}

\nid \textbf{Remark.} An analogous fact holds for \emph{rectangles} of a floorplan and
points in the graph of the corresponding \emph{R-permutation}.
It is not stated explicitly in~\cite{abp},
but follows directly from the definitions
in the same way as
 Observation~\ref{the:nei_dia} does.

\medskip

Since the neighborhood relations characterize  the S-equivalence
class, we have proved the following result.
\begin{Corollary}
\label{cor:injective}
  Two floorplans are S-equivalent if and only if they have
  the same S-permutation.
\end{Corollary}

\subsection{\boldmath{$(2\mn14\mn3, 3\mn41\mn2)$}-Avoiding  permutations}
\label{sec:patterns}

In this section we first discuss the \emph{dash notation} and
\emph{bar notation} for pattern avoidance in permutations, and then
prove several facts about $(2\mn14\mn3, 3\mn41\mn2)$-avoiding
permutations.
As 
will be proved in Section~\ref{sec:main}, these are precisely the S-permutations
obtained from floorplans.

In the classical notation, a permutation $\pi=a_1 a_2 \dots a_n$
\emph{avoids} a permutation (\emph{a pattern}) $\tau=b_1 b_2 \dots
b_k$ if there are no
$1\leq i_1 < i_2 < \ldots < i_k \leq n$
such that
$a_{i_1} a_{i_2} \dots a_{i_k}$ (a subpermutation of $\pi$) is order
isomorphic to $\tau$ ($b_x<b_y$ if and only if $a_{i_x} < a_{i_y}$).

The dash notation and the bar notation generalize the classical
notation and provide a convenient way to define more classes of
restricted permutations. For a recent survey, see~\cite{einar}.

In the dash notation, some letters corresponding to those from the
pattern $\tau$ may be required to be adjacent in the permutation
$\pi$, in the following way. If there is a dash between two letters in
$\tau$, the corresponding letters in $\pi$ may occur at any distance
from each other; if there is no dash, they must be adjacent in $\pi$.
For example, $\pi=a_1 a_2 \dots a_n$ avoids $2\mn14\mn3$ if there are
no $1\leq i < j<\ell <m$ such that $\ell=j+1$
and $a_{j} < a_{i} < a_{m} < a_{\ell}$.

In the bar notation, some letters of $\tau$ may have bars.
A permutation $\pi$ avoids a barred pattern $\tau$ if
every occurrence of the unbarred part of $\tau$ is a sub-occurrence of
$\tau$ (with bars removed).
For example, $\pi=a_1 a_2 \dots a_n$ avoids $21\bar{3}54$ if
for any $1\leq i < j<\ell <m$ such that
$a_{j} < a_{i} < a_{m} < a_{\ell}$, there exists $k$
such that
$j<k<\ell$ and $a_{i} < a_{k} < a_{m}$
(any occurrence of the pattern $2154$ is
a sub-occurrence of the pattern $21354$).

A \emph{(reduced) Baxter permutation}
is a permutation of $[n]$ satisfying the following condition:
\begin{quote}
There are no $i, j, \ell, m \in [n]$ satisfying
$i<j<\ell<m$, $\ell=j+1$,
such that \\
\mbox{\hspace{10mm}} either $\pi(j)<\pi(m)<\pi(i)<\pi(\ell)$
and
$\pi(i)=\pi(m)+1$, \\
\mbox{\hspace{10mm}} or $\pi(\ell)<\pi(i)<\pi(m)<\pi(j)$
and
$\pi(m)=\pi(i)+1$.
\end{quote}

In the dash notation, Baxter permutations are those avoiding
$(2\mn41\mn3, 3\mn14\mn2)$, and in the bar notation, Baxter
permutations are those avoiding $(41\bar{3}52, 25\bar{3}14)$
(see~\cite{gire-these} or~\cite[Sec.~7]{einar}).
As proved in~\cite{abp}, the permutations that are obtained as
R-permutations are precisely
the Baxter permutations.
It turns out that the permutations that are obtained as
S-permutations may be defined by similar conditions,
given below in Proposition~\ref{the:av2}. As in the Baxter case, these
conditions can be defined in three different ways.

\begin{Lemma}
\label{the:av}
Let $\pi$ be a permutation of $[n]$. The following conditions are equivalent:
\begin{enumerate}
  \item There are no $i, j, \ell, m \in [n]$ such that $i<j<\ell<m$, $\ell=j+1$, $\pi(j)<\pi(i)<\pi(m)<\pi(\ell)$, $\pi(m)=\pi(i)+1$.
  \item In the dash notation, $\pi$ avoids $2\mn 14\mn 3$.
  \item In the bar notation, $\pi$ avoids $21\bar354$.
\end{enumerate}
\end{Lemma}

Fig.~\ref{fig:three_cond} illustrates these three conditions.
The rows (respectively, columns) marked by dots in parts~(1) and~(2) denote adjacent rows (respectively, columns).
The shaded area in part (3) does not contain points of the graph.

\begin{figure}[ht]
$$\resizebox{120mm}{!}{\includegraphics{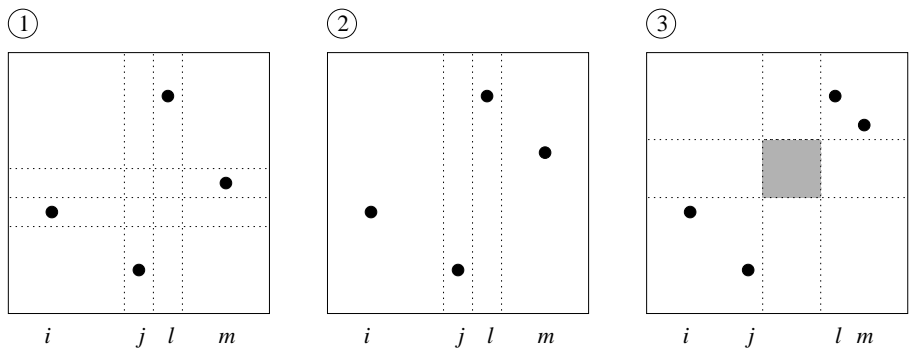}}$$
\caption{Three ways to define permutations avoiding
$2\mn14\mn3$.}
\label{fig:three_cond}
\end{figure}

\begin{proof}
It is clear that $3 \Rightarrow 2 \Rightarrow 1$: the four points
displayed in Fig.~\ref{fig:three_cond}(1)
form an occurrence of  the pattern of  Fig.~\ref{fig:three_cond}(2),
and the four points
displayed in Fig.~\ref{fig:three_cond}(2) form an occurrence of  the
pattern of Fig.~\ref{fig:three_cond}(3).

Conversely, let us prove that if a permutation $\pi$ contains the pattern $21\bar354$, then there exist $i', j', \ell', m'$ as in the first condition.
Assume that there are
$i<j<\ell<m$ such that $\pi(j)<\pi(i)<\pi(m)<\pi(\ell)$,
and there is no $k$ such that $j<k<\ell$ and $\pi(i)<\pi(k)<\pi(m)$.
Let $j'$ be the maximal number for which $j \leq j' < \ell$ and
$\pi(j')<\pi(i)$. Let $\ell'=j'+1$. Then $\pi(\ell')>\pi(m)$, and we
have a pattern $2 \mn 14 \mn 3$ with $i, j', \ell', m$.

Furthermore, let $i'$ be the number satisfying
 $ i' < j'$
and $\pi(i) \leq \pi(i')<\pi(m)$, for which $\pi(i')$ is the maximal
possible. Let $m' = \pi^{-1}(\pi(i')+1)$. Then
$m'>\ell'$
and $\pi(m')=\pi(i')+1$, and, thus, the first condition holds with
$i', j', \ell', m'$.
\end{proof}

A similar result holds for permutations that avoid $3\mn41\mn2$.
Therefore, the following proposition holds.

\begin{Proposition}
\label{the:av2}
Let $\sigma$ be a permutation of $[n]$. The following statements are equivalent:
\begin{enumerate}
  \item There are no $i, j, \ell, m \in [n]$ satisfying $i<j<\ell<m$, $\ell=j+1$, such that \\
\mbox{\hspace{10mm}} either
$\sigma(j)<\sigma(i)<\sigma(m)<\sigma(\ell)$
and $\sigma(m)=\sigma(i)+1$, \\
\mbox{\hspace{10mm}} or $\sigma(\ell)<\sigma(m)<\sigma(i)<\sigma(j)$
and $\sigma(i)=\sigma(m)+1$.
  \item In the dash notation, $\sigma$ avoids $2\mn 14\mn 3$ and $3\mn 41\mn 2$.
  \item In the bar notation, $\sigma$ avoids $21\bar354$ and $45\bar312$.
\end{enumerate}
\end{Proposition}

\begin{Corollary}
  The group of symmetries of the square leaves invariant the set of
  $(2\mn 14\mn 3, 3\mn 41\mn 2)$-avoiding \ps.
\end{Corollary}
\begin{proof}
  The second description in Proposition~\ref{the:av2} shows that the
  set of $(2\mn 14\mn 3, 3\mn
41\mn 2)$-avoiding \ps\ is closed under reading the \ps\ from
right to left. The first (or third) description shows that it is
invariant under taking inverses, and these two transformations
generate the symmetries of the square.
\end{proof}

We shall also use the following fact.

 \begin{Lemma}
\label{the:av_f}
Let $\sigma$ be a $(2\mn 14\mn 3, 3\mn 41\mn 2)$-avoiding permutation of $[n]$.
Then no point in the  graph of $\sigma$ has several  NW-neighbors \emm
and, several  NE-neighbors.
Similar statements hold for other pairs of adjacent diagonal directions.
\end{Lemma}

\begin{proof}
Assume that $N_i=(i, \si(i))$ has several NW-neighbors and several NE-neighbors.
Let $i'$ be the maximal number for which $N_{i'}$ is a NW-neighbor of $N_i$, and let $N_j$ be another NW-neighbor of $N_i$.
Then we have $j<i'$ and $\sigma(i) < \sigma(j) < \sigma (i')$.
We conclude that $i'=i-1$:
otherwise $\sigma(i'+1)<\sigma(i)$ and, therefore, $j, i', i'+1, i$
form the forbidden pattern $3 \mn 41 \mn 2$,
which is a contradiction.

Similarly, if $i''$ is the minimal number for which $N_{i''}$ is a NE-neighbor of $N_i$, then $i''=i+1$.
Let $N_k$ be another NE-neighbor of $N_i$. We have $\sigma(i) < \sigma(k) < \sigma (i+1)$.

Assume without loss of generality that $\sigma(i-1)<\sigma(i+1)$.
Now, if $\sigma(j)<\sigma(k)$,
then $j, i, i+1, k$ form the forbidden pattern $2 \mn 14 \mn 3$;
and if
$\sigma(k)<\sigma(j)$, then $j, i-1, i, k$ form the forbidden pattern
$3 \mn 41 \mn 2$,
which is, again, a contradiction.
\end{proof}

\subsection{S-permutations coincide with \boldmath{$(2\mn 14\mn 3, 3\mn 41\mn 2)$}-avoiding permutations}
\label{sec:main}

By Corollary~\ref{cor:injective},
the map $S$ induces an injection from S-equivalence classes of floorplans to permutations. Here, we characterize the image of $S$.
\begin{Theorem}
  The map $S$ induces
a bijection  between
  S-equivalence classes of floorplans of size $n+1$ and
  $(2\mn14\mn3$, $3\mn41\mn2)$-avoiding permutations of size $n$.
\end{Theorem}

The proof involves two steps: In
Proposition~\ref{the:bij1}
we prove that
all S-permutations are $(2\mn14\mn3$, $3\mn41\mn2)$-avoiding.
Then, in
Proposition~\ref{the:bij2},
we show that for any
$(2\mn14\mn3$, $3\mn41\mn2)$-avoiding permutation $\sigma$ of $[n]$,
there exists a floorplan $P$ with $n$ segments such that $S(P)=\sigma$.

Recall that a horizontal segment has at most one left-neighbor and at
most one right-neighbor, and a vertical segment has at most one
below-neighbor and at most one above-neighbor.
This translates as follows in terms of S-\ps.

\begin{Observation}
\label{the:nei_one}
Let $I_i$ be a segment in a floorplan $P$, and let $N_i$ be the corresponding point in the graph of $S(P)$.
If $I_i$ is a horizontal segment, then the point $N_i$ has at most one NE-neighbor and at most one SW-neighbor. Similarly, if $I_i$ is a vertical segment, then $N_i$ has at most one SE-neighbor and at most one NW-neighbor.
\end{Observation}

\begin{Proposition}
  \label{the:bij1}
Let $P$ be a floorplan. Then $S(P)$
avoids $2\mn14\mn3$ and  $3\mn41\mn2$.
\end{Proposition}

\begin{proof}
By Proposition~\ref{obs:sym-S}, the image of $S$ is invariant by all
symmetries of the square. Hence it suffices to prove that $\si=S(P)$
avoids $2\mn14\mn3$.

Assume that $\sigma$ contains  $2\mn 14\mn 3$.
By Lemma~\ref{the:av},
there exist $i<j<\ell<m$, $\ell=j+1$ such that
$\sigma(j)<\sigma(i)<\sigma(m)<\sigma(\ell)$ and $\sigma(m)=\sigma(i)+1$
(see Fig.~\ref{fig:impos}(1)).
We claim that the four segments $I_i$, $I_j$, $I_{\ell}$, $I_m$ are vertical.

Consider $I_j$. The point $N_{\ell}$ is a NE-neighbor of $N_j$.
Consider the set $\{x: \ x > \ell, \ \sigma(j) < \sigma(x) < \sigma(\ell)
\}$. This set is not empty since it contains $m$. Let $p$ be the
smallest element in this set. Then $N_p$ is a NE-neighbor of
$N_j$. Thus, $N_j$ has at least two NE-neighbors, $N_{\ell}$ and
$N_p$. Therefore, $I_j$ is vertical by
Observation~\ref{the:nei_one}. In a similar way one can show that
$I_i$, $I_{\ell}$, $I_m$ are also vertical.

\begin{figure}[ht]
\begin{center}
$\resizebox{125mm}{!}{\includegraphics{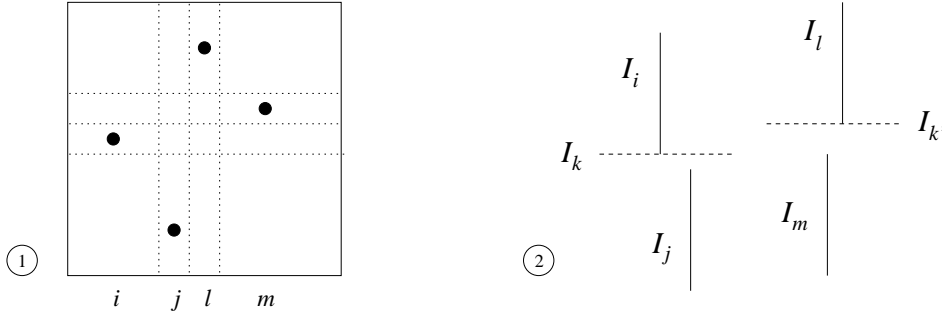}}$
\end{center}
\caption{The pattern $2\mn 14\mn 3$ never occurs in an S-permutation.}
\label{fig:impos}
\end{figure}

By Observation~\ref{the:nei_dia} we have that:
$I_j \es I_i$, $I_m \es I_{\ell}$;
$I_i \ew I_\ell$, $I_j\ew I_m$,
$I_i \ew I_m$,
$I_j \ew I_{\ell}$.
Moreover, the last two relations are neighborhood relations.
Let $I_k$ be the below-neighbor of $I_i$, and let $I_{k'}$ be the
below-neighbor of $I_{\ell}$  (see
Fig.~\ref{fig:impos}~(2)). The segments $I_k$ and $I_{k'}$ are horizontal.
If the line supporting $I_k$ is (weakly) lower than the line
supporting $I_{k'}$, then $I_j$
(which is below $I_i$)
cannot be a left-neighbor of
$I_{\ell}$ since the interiors of their vertical projections do not
intersect.
Similarly, if the line supporting $I_k$ is
 higher than the line supporting $I_{k'}$, then $I_i$ cannot
be a left-neighbor of $I_m$.
We have thus reached a contradiction, and $\sigma$ cannot contain $2\mn 14\mn 3$.
\end{proof}

\begin{Proposition}
  \label{the:bij2}
For each $(2\mn14\mn3, 3\mn41\mn2)$-avoiding permutation $\sigma$ of $[n]$, there exists a floorplan $P$ with $n$ segments such that $S(P)=\sigma$.
\end{Proposition}

\begin{proof}
We construct $P$ on the graph of $\sigma$.
The boundary of the graph is also the boundary of $P$.
For each point $N_i = (i, \sigma(i))$ of the graph, we draw a segment $K_i$ passing through $N_i$ according to certain rules.
We first determine the direction of the segments $K_i$ (Paragraph~A
below),
and then the coordinates of their endpoints (Paragraph~B).
We prove that we indeed obtain a floorplan (Paragraph~C),
and that its S-permutation is $\sigma$ (Paragraph~D).
This is probably one of the most involved proofs of the paper.

\medskip

\paragraph{\bf A. Directions of the segments $K_i$}
\label{par:dir}\mbox{}

\nid
Let $N_i = (i, \sigma(i))$ be a point in the graph of $\sigma$.
Our first two  rules are forced by Observation~\ref{the:nei_one}. They
are illustrated in Fig.~\ref{fig:rules0} (no point of the graph
lies in the shaded areas):
\begin{itemize}
  \item If $N_i$ has several NW-neighbors or several    SE-neighbors,
then $K_i$ is horizontal.
  \item If $N_i$ has several SW-neighbors or several NE-neighbors,
then $K_i$ is vertical.
\end{itemize}

\nid By Lemma~\ref{the:av_f}, these two rules
never apply simultaneously to the same point $N_i$.
If one of them applies, 
we say that $N_i$ is a \emph{strong point}.
Otherwise, 
 $N_i$ is a \emph{weak point}.
This means that $N_i$ has at most one neighbor in each direction.

\begin{figure}[ht]
\begin{center}
$\resizebox{140mm}{!}{\includegraphics{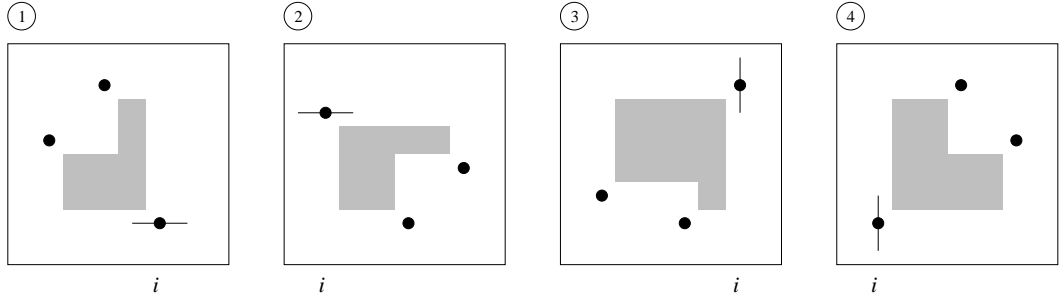}}$
\end{center}
\caption{Rules for determining the direction of the segment $K_i$ passing through a strong point.}
\label{fig:rules0}
\end{figure}

We claim that if $N_{i}$ and $N_{j}$ are weak points, then they
are in adjacent rows if and only if they are in adjacent columns. Due
to symmetry, it suffices to show the \emph{if} direction.
Let $N_{i}$ and $N_{i+1}$ be weak points, and assume without loss of
generality that $\sigma(i) < \sigma(i+1)$.
If $\sigma(i+1) - \sigma(i) > 1$, then there are points of the graph
of $\sigma$
between the rows that contain
$N_i$ and $N_{i+1}$;
thus, either $N_{i}$ has at least two NE-neighbors or
$N_{i+1}$ has at least two SW-neighbors,
which means that one of them at least is strong. Hence $\si(i+1)=\si(i)+1$.

Thus, weak points appear as ascending or descending sequences of
adjacent neighbors:
$N_i, N_{i+1},
\ldots, N_{i+\ell}$ with  $\sigma(i)=\sigma(i+1)-1=
 \cdots = \sigma(i+\ell)-\ell$ or
 $\sigma(i)=\sigma(i+1)+1
= \cdots =\sigma(i+\ell)+\ell$.
Note that a weak point $N_i$ can be isolated.

For weak points, the direction of the corresponding segments is determined as follows:

\begin{itemize}
  \item If $N_i, N_{i+1}, \ldots, N_{i+\ell}$ is a maximal
    \emph{ascending} sequence of weak points,
then the directions of $K_i, K_{i+1},
\ldots,$ $K_{i+\ell}$ are chosen in such a way
that $K_j$ and $K_{j+1}$ are never both horizontal, for $i
\leq j < i+\ell$. Hence several choices are possible
(this multiplicity of choices is consistent with the fact that all
S-equivalent floorplans give the same \p).
\item If $N_i, N_{i+1},
\ldots, N_{i+\ell}$ is a maximal \emph{descending} sequence of weak
points, then the directions of $K_i, K_{i+1},
\ldots,$ $K_{i+\ell}$ are chosen
in such a way that $K_j$ and $K_{j+1}$ are never both vertical, for $i
\leq j < i+\ell$.
\end{itemize}

In particular, for an isolated weak point $N_i$, the direction of
$K_i$ can be chosen arbitrarily.

\medskip
\paragraph{\bf  B. Endpoints of the segments $K_i$}
\label{par:endpoints}\mbox{}

\nid Once the directions of all $K_i$'s are chosen,
their endpoints are set as follows
(see Fig.~\ref{fig:rules} for an illustration):
\begin{itemize}
  \item If $K_i$ is vertical (which implies that $N_i$
 has at most one NW-neighbor and at most one SE-neighbor):

\begin{itemize}
\item If $N_i$ has a NW-neighbor $N_j$, then the upper endpoint of $K_i$ is
set to be 
 $(i, \sigma(j))$. We say that \emph{$N_j$ bounds
  $K_i$ from above}.
Otherwise (if $N_i$ has no NW-neighbor), $K_i$ reaches the upper side
of the boundary.

\item If $N_i$ has a SE-neighbor $N_k$, then the lower endpoint of $K_i$ is
$(i, \sigma(k))$. We say that \emph{$N_k$ bounds $K_i$ from below}.
Otherwise,
$K_i$ reaches the lower side of the boundary.
\end{itemize}
\item If $K_i$ is horizontal (which
implies that
$N_i$  has at most one SW-neighbor and at most one NE-neighbor):
\begin{itemize}
\item
If $N_i$ has a SW-neighbor $N_j$, then the left endpoint of $K_i$ is
$(j, \sigma(i))$. We say that \emph{$N_j$ bounds $K_i$ from the left}.
Otherwise, 
$K_i$ reaches the left side of the boundary.

\item
If $N_i$ has a NE-neighbor $N_k$, then the right endpoint of $K_i$ is
$(k, \sigma(i))$. We say that \emph{$N_k$ bounds $K_i$ from the right}.
Otherwise, 
$K_i$ reaches the right side of the boundary.
\end{itemize}
\end{itemize}

\begin{figure}[ht]
$$\resizebox{95mm}{!}{\includegraphics{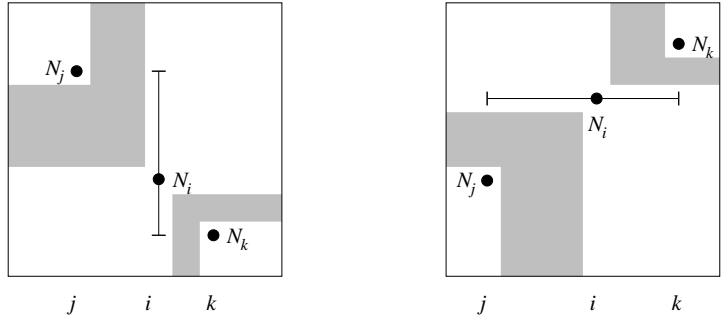}}$$
\caption{Determining the endpoints of the segment $K_i$: the points $N_j$ and $N_k$ bound the segment $K_i$.}
\label{fig:rules}
\end{figure}

Fig.~\ref{fig:alg} presents an example of the whole construction:
in Part 1, the directions are determined for strong (black) points,
and chosen for weak (gray) points; in Part 2, the endpoints are
determined and a floorplan is obtained. Notice that $\sigma$
is the S-permutation
associated with the floorplan
$P$ of Fig.~\ref{fig:ex1}, but here we have obtained a different
floorplan, $P'$.
We leave it to the reader to check that another choice of
directions of segments passing through weak points leads to $P$.

The question of when $S(P)=S(P')$ will be studied in
Section~\ref{sec:same-S}.

\begin{figure}[ht]
$$\resizebox{110mm}{!}{\includegraphics{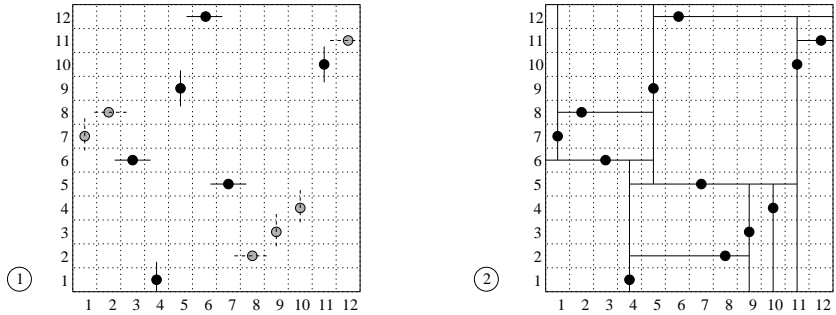}}$$
\caption{Constructing a floorplan from a $(2\mn14\mn3, 3\mn41\mn2)$-avoiding permutation.}
\label{fig:alg}
\end{figure}

\nid{\bf Remark.} Using dynamic programming, one can determine in
linear time the values
$$
m_i=\max \{k <i : \si(k)<\si(i)\}\cup\{0\} \quad \hbox{and }\quad
M_i=\max \{k <i : \si(k)>\si(i)\}\cup\{0\}.
$$
By applying this procedure to $\si$ and to the \ps\ obtained by applying
to $\si$ a symmetry of the square, one can decide in linear time, for
each point $N_i$ of $\si$, if it has one or several NW-neighbours and locate one of them. This
implies that the above construction of a floorplan starting from a
$(2\mn14\mn3, 3\mn41\mn2)$-avoiding \p\ can be done in linear time.

\medskip
\paragraph{\bf C. The construction 
 indeed determines a floorplan}\label{par:indeed}\mbox{}

\nid In order to prove this, we need to show that
two segments never cross,
and that the endpoints of any segment $K_i$ are contained in segments
perpendicular to $K_i$ (unless they lie on the boundary).
The following observation will simplify some of our proofs.
\begin{Observation}
\label{obs:sym-reverse}
  Let $\si$ be a $(2\mn14\mn3, 3\mn41\mn2)$-avoiding \p, and let $\si'$ be obtained by applying a
  rotation $\rho$ (a quarter turn or a half-turn,
clockwise or counterclockwise)
to (the graph of)
  $\si$.  If $P$ is a configuration of segments obtained from $\si$ by
  applying the rules of Paragraphs A and B above, then $\rho(P)$ can be
  obtained from $\si'$ using those rules.
\end{Observation}
To prove this, it suffices to check that the rules are invariant by a
$90^\circ$ rotation, which is immediate\footnote{That the construction
  has the other symmetries of Proposition~\ref{obs:sym-S}
is also true, but less obvious.
We shall only use Observation~\ref{obs:sym-reverse}.}.

\medskip

\noindent{\bf C.1.  Let $K_i$ be a vertical (respectively, horizontal) segment, and let $N_j$ and $N_k$ be the points that bound it. Then the segments $K_j$ and $K_k$ are horizontal (respectively, vertical).}

\nid
Thanks to Observation~\ref{obs:sym-reverse},
it suffices to prove this claim for a \emph{vertical} segment $K_i$ and for the point $N_j$ that bounds it \emph{from above}. We need to prove that $K_j$ is a horizontal segment.

We have $j<i$ and $\sigma(i) < \sigma(j)$, and, since $N_{j}$ is a NW-neighbor of $N_i$, there is no $\ell$ such that $j<\ell<i$ and $\sigma(i) < \sigma(\ell) < \sigma(j)$.
Furthermore, there is no $\ell$ such that $j<\ell<i$, $\sigma(j) < \sigma(\ell)$, or such that
$\ell<j$, $\sigma(i) < \sigma(\ell) < \sigma(j)$: otherwise $N_i$
would have several
NW-neighbors and, therefore, $K_i$ would be horizontal. Now, if
$i-j>1$,
then there
exists $\ell$ such that $j<\ell<i$, $\sigma(\ell) < \sigma(i)$;
and if $\sigma(j)-\sigma(i)>1$, then there
exists $m$ such that $i<m$,
$\sigma(i) < \sigma(m) < \sigma(j)$. In both cases $N_j$ has several
SE-neighbors, and, therefore, $K_j$ is horizontal as claimed.

It remains to consider the case where $j=i-1$ and $\sigma(j)
=\sigma(i)+1$.
If the point $N_i$ is strong, then (since $K_i$ is vertical) it has
several NE-neighbors or
several SW-neighbors. Assume without loss of generality that $N_i$ has
several
NE-neighbors. Let $\ell$ be the minimal number such that $N_{\ell}$ is
a NE-neighbor of $N_i$, and let $N_m$ be another NE-neighbor of
$N_i$. Then we have $\sigma(i-1) < \sigma(m) < \sigma(\ell)$ and
$\sigma(\ell-1) \leq \sigma(i)$. However, then $i-1, \ell-1, \ell, m$
form a forbidden pattern $2\mn 14 \mn 3$. Therefore, $N_i$ is a weak
point.
Clearly, $N_{i-1}$ as a unique SE-neighbor (which is $N_i$). Its NE-
and SW-neighbors coincide with those of $N_i$, so that there is at
most one of each type. Thus if $N_{i-1}$ is strong, it has several NW-neighbors, and $K_{i-1}$ is horizontal, as claimed.
If $N_{i-1}$ is weak,
then
the rules that determine the direction of the
segments passing through
(descending)
weak points implies that  $K_{i-1}$ and $K_i$ cannot
be both vertical. Therefore, $K_j=K_{i-1}$ is horizontal, as claimed.

\medskip

\noindent{\bf C.2. If $N_j$ and $N_k$ are the points that bound the segment $K_i$, then the segments $K_j$ and $K_k$ contain the endpoints of $K_i$}

\nid
Thanks to Observation~\ref{obs:sym-reverse},
it suffices to show that if $K_i$ is a vertical segment and $N_j$
bounds it from above, then $K_j$ (which is horizontal as shown in
Paragraph C.1 above)
contains the point $(i, \sigma(j))$.
We saw in Paragraph C.1
that in this situation there is no $\ell$ such that $j<\ell <i$, $\sigma(j) < \sigma(\ell)$. This means that there is no point $N_\ell$ that could bound $K_j$ from the right before it reaches $(i, \sigma(j))$.

\medskip
\nid{\bf C.3. Two segments $K_{i}$ and $K_{j}$ cannot cross}\label{par:cross}\mbox{}

\nid Assume
that $K_{i}$ and $K_{j}$ cross.
Assume without loss of generality that $K_{i}$ is vertical and $K_{j}$ is horizontal,
so that their crossing point is $(i, \sigma(j))$.
We have either $i<j$ or $j<i$, and $\sigma(i)<\sigma(j)$ or $\sigma(j)<\sigma(i)$.
Assume without loss of generality $j<i$ and $\sigma(i)<\sigma(j)$.
Then $N_j$ is to the NW of $N_i$.
The ordinate of the (unique) NE-neighbor of $N_i$ is hence at most
$\sigma(j)$. By construction, the upper point of $K_i$ has ordinate at most
$\sigma(j)$, while
$K_j$ lies at ordinate $\sigma(j)$, and thus $K_{i}$ and $K_{j}$ cannot cross.

We have thus proved that our construction indeed gives a floorplan.
Let us finish with an observation on joins of segments of this floorplan,
which follows from Paragraph C.2
and is illustrated below.

\begin{Observation}
  \label{the:join-types}
Suppose that a vertical segment $K_i$ and a horizontal segment $K_j$
join at the point $(i, \sigma(j))$.
Then:
\begin{itemize}
  \item If the join of $K_i$ and $K_j$ is of the type $\downvdash$, then $i>j$.
  \item If the join of $K_i$ and $K_j$ is of the type $\upvdash$, then $i<j$.
  \item If the join of $K_i$ and $K_j$ is of the type $\rightvdash$, then $\sigma(i)<\sigma(j)$.
  \item If the join of $K_i$ and $K_j$ is of the type $\leftvdash$, then $\sigma(i)>\sigma(j)$.
  \end{itemize}
\medskip
  $$\includegraphics[width=90mm]{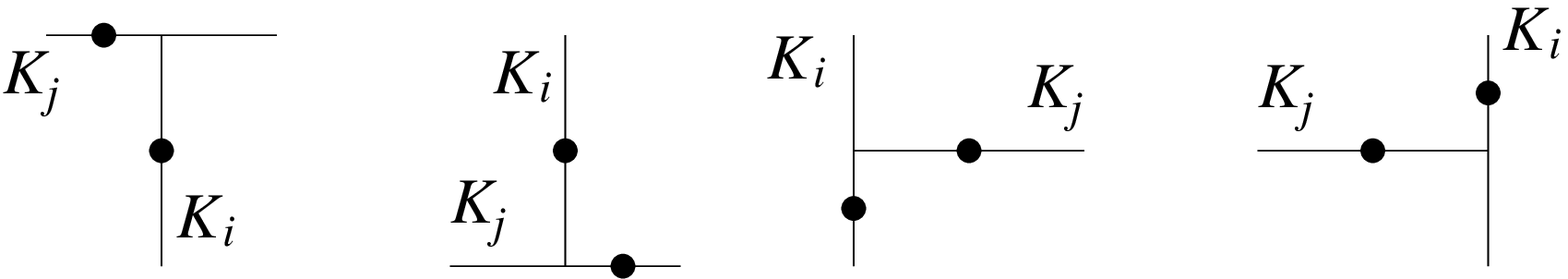}$$
\end{Observation}

\medskip\paragraph{\bf D. For any floorplan $P$ obtained by the
  construction described above, $S(P)=\sigma$}
\mbox{}

This (concluding) part of the proof is 
given in Appendix~\ref{app:sur}.
\end{proof}

\section{Relations between the R- and S-permutations}
\label{sec:seg_rec}
In this section we prove that
if two floorplans are R-equivalent, 
 they are S-equivalent
(that is,
 the R-equivalence 
refines the S-equivalence).
In fact,  we give a simple graphical
way to construct $S(P)$ from $R(P)$,
which also shows that $S(P)$ and
$R(P)$ taken together form the \emm complete, Baxter \p\ associated
with the reduced Baxter \p\ $R(P)$.
Finally, we characterize
 the R-equivalence classes that belong to the same S-equivalence class.

\subsection{Constructing $S(P)$ from $R(P)$}
\label{sec:r2s}

Let $P$ be a floorplan of size $n+1$. We draw the graphs of
$\rho = R(P)$ and $\sigma = S(P)$ on the same diagram in the
following way (Fig.~\ref{fig:ex1_rec}).
For the graph of $\rho$ we use
an $(n+1)\times(n+1)$ square whose columns and rows are numbered by
$1, 2, \dots, n+1$. The points of the graph of $\rho$
are placed
at the centers of these squares, and these points
are black.
For the graph of $\sigma$ we use the grid lines of the same drawing,
when the $i$th vertical (respectively, horizontal) line is the grid
line between the $i$th and the $(i+1)$st columns (respectively,
rows). The point $(i, \sigma(i))$ 
is placed at
the intersection of the $i$th vertical grid line and the $j$th
horizontal grid line, where $j=\sigma(i)$.
Such points are white. The whole drawing
is called the \emph{combined diagram} of $P$.
Note that the extreme (rightmost, leftmost, etc.) grid lines are not
used.

\begin{figure}[ht]
$$\resizebox{165mm}{!}{\includegraphics{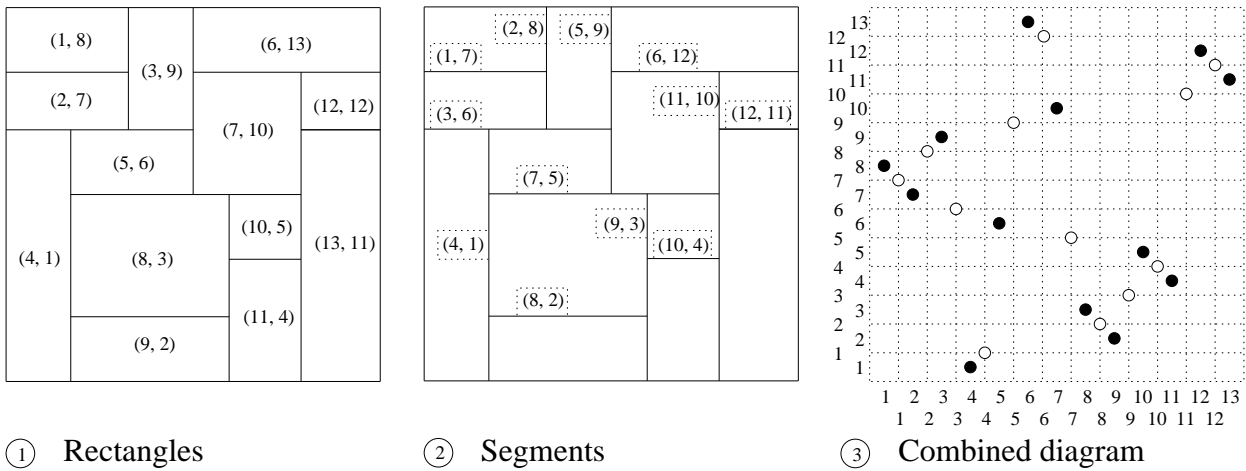}}$$
\caption{The floorplan $P$ from Fig.~\ref{fig:ex1}: (1) The labeling of rectangles; (2) The labeling of segments; (3) The combined diagram:
$R(P) = 8 \ 7 \ 9 \ 1 \ 6 \ 13 \ 10 \ 3 \ 2 \ 5 \ 4 \ 12 \ 11$ (black points)
together with
$S(P) = 7 \ 8 \ 6 \ 1 \ 9 \ 12 \ 5 \ 2 \ 3 \ 4 \ 10 \ 11$ (white points).}
\label{fig:ex1_rec}
\end{figure}

\begin{Definition}
  \label{the:bax}
Let $\rho$ be a Baxter permutation of $[n+1]$. For 
$1 \leq i \leq n$, define $j_i$ as follows:
\begin{itemize}
  \item if $\rho(i) < \rho (i+1)$, then
$\displaystyle
j_i= \max \{\rho(k), k \le i \hbox{ and } \rho(k) <\rho(i+1)\},
$
  \item if $\rho(i) > \rho (i+1)$, then
$\displaystyle
j_i= \max \{\rho(k), k \ge i+1 \hbox{ and } \rho(k) <\rho(i)\}.
$
\end{itemize}
\end{Definition}
\nid The definition of Baxter \ps\ implies that
\begin{itemize}
  \item if $\rho(i) < \rho (i+1)$, $k \ge i+1$ and $\rho(k)> \rho(i)$,
    then $\rho(k) >\rho(j_i)$,
  \item if $\rho(i) > \rho (i+1)$,  $k \le i$ and $\rho(k)> \rho(i+1)$,
    then $\rho(k) >\rho(j_i)$.
\end{itemize}

\begin{Theorem}
\label{the:r2s}
Let $P$ be a floorplan of size $n+1$, and let $\rho =
R(P)$.
Then $S(P) = (j_1, j_2, \dots, j_n)$, where $j_i$ is defined in
Definition~\ref{the:bax}.
In particular, R-equivalent floorplans are also S-equivalent.
\end{Theorem}

Returning to the original papers on Baxter \ps\ (see for instance~\cite[Thm.~2]{boyce67}, or the
definition of complete \ps\ in~\cite[p.~180]{boyce81})
this means that the
combined diagram forms a (complete) Baxter permutation $\pi$. The points of
$R(P)$ form the reduced Baxter \p\ $\pi_o$ associated with $\pi$, and the
points of $S(P)$ are those
that are deleted from $\pi$  when constructing $\pi_o$.

\begin{proof}
Let $i \in [n]$. Denote $\sigma=S(P)$ and $j=\sigma(i)$.
Then the segment $I_i$ labeled $i$ in the $\tnw$~order, is labeled $j$
in the $\tsw$~order. We denote by $A_k$ (resp. $B^k$) the $k$th \emm
rectangle, in the $\tnw$- (resp. $\tsw$-) order.
We wish to prove that $j=j_i$.

Assume first that $I_i$ is horizontal.
By Observation~\ref{the:r2s_1}, the rightmost 
 rectangle whose lower side is contained in $I_i$ is $A_i$, and the
 leftmost
 rectangle whose upper side is contained in $I_i$ is $A_{i+1}$
 (Fig.~\ref{fig:r2s_3}).

By definition of $\rho$, we have $A_k=B^{\rho(k)}$ for all $k$.
By symmetry, since
$I_i$
is the $j$th segment in the $\tsw$ order,
the rightmost 
 rectangle whose
 upper side is contained in $I_i$ is $B^{j}$, and the leftmost
 rectangle whose lower side is contained  in $I_i$ is  $B^{j+1}$.
There holds
$A_{i+1} \tsw B^{j} \esw B^{j+1} \tsw A_i$ and $B^{j+1} \tnw A_i \enw
A_{i+1} \tnw B^{j}$.
By definition of $\rho=R(P)$, this means
$\rho(i+1) \leq j < j +1 \leq \rho(i)$ and $\rho^{-1}(j+1) \leq i < i+1 \leq \rho^{-1}(j)$.
This shows that $j$ coincides with the value $j_i$ of Definition~\ref{the:bax}
(for the case $\rho(i)>\rho(i+1)$).

The case where $I_i$ is vertical is similar, and corresponds to an
ascent in $\rho$.
\end{proof}

\begin{figure}[ht]
\begin{center}
\resizebox{50mm}{!}{\includegraphics{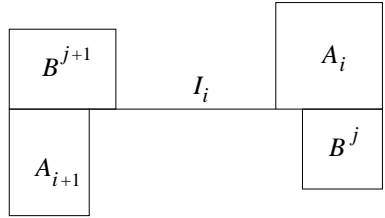}}
\caption{Illustration of the proof of Theorem~\ref{the:r2s}.}
\label{fig:r2s_3}
\end{center}
\end{figure}

The symmetry in the definition of $j_i$ makes the following property
obvious, without going through floorplans.
\begin{Corollary}
Let $P$ be a floorplan and let $\rho=R(P)$ be the corresponding
Baxter \p.
  Let us abuse notation by denoting $S(\rho):=S(P)$. If
  $\rho'$ is obtained by applying to $\rho$ a symmetry of the square,
  then the same symmetry, applied to $S(\rho)$, gives $S(\rho')$.
\end{Corollary}

\nid {\bf Remark.} The combined diagram is actually the R-permutation
of a floorplan of size $2n+1$. Indeed, let $P$ be a
floorplan of size $n+1$.
If we
inflate segments of $P$ into narrow rectangles,
we obtain a new
floorplan of size $2n+1$, which we denote by $\tilde{P}$ (Fig.~\ref{fig:inflated-segments}).
Observe that a rectangle of $\tilde{P}$ corresponding to a rectangle
$A$ of $P$ has
a unique above (respectively, right, below, left) neighbor,
which corresponds to the segment of $P$
that contains the above (respectively, right, below, left) side of~$A$.

\begin{figure}[ht]
\begin{center}
\resizebox{50mm}{!}{\includegraphics{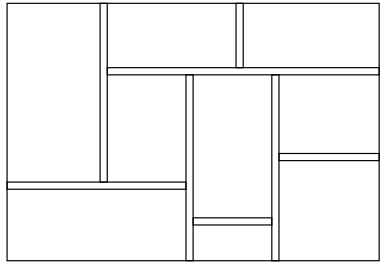}}
\caption{Inflating the segments of a floorplan.}
\label{fig:inflated-segments}
\end{center}
\end{figure}

 It follows from Observation~\ref{the:r2s_1} and Fig.~\ref{fig:r2s_1} that
 the $\tnw$ order in $\tilde P$ is $A_1  I_1 A_2  I_2
 \ldots A_n I_n A_{n+1}$. It is thus obtained by shuffling the
 $\tnw$ orders for rectangles and segments of $P$. Symmetrically, the
 $\tsw$ order in  $\tilde P$ is $A_{\rho^{-1}(1)}  I_{\sigma^{-1}(1)}
 \cdots A_{\rho^{-1}(n)}  I_{\sigma^{-1}(n)} A_{\rho^{-1}(n+1)} $.
 Thus the combined diagram of $R(P)$ and $S(P)$, as in
Fig.~\ref{fig:ex1_rec}, coincides with the graph of
$R(\tilde{P})$.

\subsection{Floorplans that produce the same S-permutation}
\label{sec:same-S}
In this section we characterize in terms of their R-\ps\ the
floorplans that have the same S-\p.
This will play a central role in the enumeration of S-\ps.

We first describe the floorplans whose S-permutation is
$123\dots n$. Such floorplans will be called
\emph{ascending F-blocks}.\footnote{The letter F refers to Fibonacci,
  for reasons that will be explained further down.} 
It is easy to see that in an ascending F-block,
all vertical segments extend from the lower to the upper side of the
boundary,
and there is at most one horizontal segment between a pair of adjacent vertical segments
(this can be shown inductively, by noticing that at most one horizontal
segment starts from the left side of the bounding rectangle). See
Fig.~\ref{fig:fib0}.
Conversely, every floorplan of this type has S-permutation
$123\ldots n$.
Therefore, an ascending F-block consists of several rectangles that
extend from the lower to the upper side of the boundary,
some of them being split into two sub-rectangles by a horizontal segment.
 The corresponding R-permutations are
those that satisfy
$|\rho(i)-i|\leq 1$ for all $1\leq i \leq n+1$. The number of
 ascending F-blocks of size $n+1$ (and, therefore, the number of such
 permutations) is the Fibonacci number $F_{n+1}$ (where
 $F_0=F_1=1$). 

\begin{figure}[h]
$$\resizebox{140mm}{!}{\includegraphics{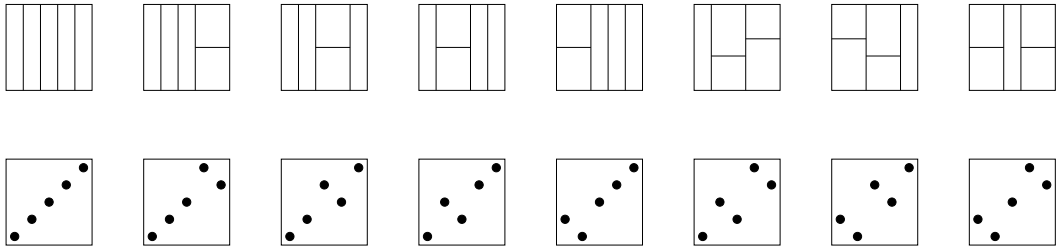}}$$
\caption{The 8 ascending F-blocks for $n=4$, and their R-permutations.}
\label{fig:fib0}
\end{figure}

A similar observation holds for the floorplans whose
S-permutation is $n\dots 321$. Such floorplans are called
\emph{descending F-blocks}. In descending F-blocks,
all horizontal segments extend from the left
side
to the right side of the boundary,
and there is at most one vertical segment between a pair of adjacent horizontal segments.
In other words, descending F-blocks consist of several rectangles that
extend from the left to the right side of the boundary,
some of them being split into two sub-rectangles by a vertical segment.
The corresponding R-permutations
are characterized by the condition
$|\rho(i) - (n + 2 -i)| \leq 1$
for all $1\leq i \leq n+1$.

For an F-block $F$, the size of $F$ (that is, the number of
rectangles) will be denoted by $|F|$.
If $|F|=1$, we say that $F$ is a \emph{trivial} F-block.
Note that if $|F|\leq 2$, then $F$ is both ascending and descending,
while if $|F|\geq 3$, then its type (ascending or descending) is
uniquely determined.

Let $P$ be a floorplan.
We define an \emph{F-block in $P$} as a set of rectangles of $P$ whose union is an F-block, as defined above.
In other words, their union is a rectangle,
and the S-permutation of the induced subpartition is either $123\dots$ or $\dots 321$.
The F-blocks of $P$ are partially ordered by inclusion.
Since segments of $P$ do not cross, a rectangle in $P$ belongs
precisely to one maximal F-block (which may be of size $1$).
So there is a uniquely determined partition of $P$ into maximal
F-blocks (Fig.~\ref{fig:fblocks}, left).

A \emph{block} in a \p\ $\rho$ is an interval
$[i,j]$ such that the values $\{\rho(i), \ldots, \rho(j)\}$ also form
an interval~\cite{albert-atkinson}. By extension, we also call a block
the corresponding set of points in the graph of $\rho$.
Consider $\ell$ rectangles in $P$ that form an ascending (respectively,
descending) F-block. By Observation~\ref{the:r2s_1}
and the analogous statement for the $\tsw$ order, these $\ell$ rectangles
form an interval
in the $\tnw$ and $\tsw$~orders. Hence the corresponding
$\ell$ points of the graph of $R(P)$ form a block,
and their inner order is isomorphic to a permutation $\tau$ of
$[\ell]$ that satisfies $|\tau(i)-i|\leq 1$ (respectively, $|\tau(i) -
(\ell +1 -i)| \leq 1$) for all $1\leq i \leq \ell$.

The converse is also true:
If $\ell$ points of the graph of $R(P)$ form an $\ell \times \ell$
block, and their inner order is isomorphic to a permutation
$\tau$ of $[\ell]$ that satisfies $|\tau(i)-i|\leq 1$ (respectively,
$|\tau(i) - (\ell +1 -i)| \leq 1$) for all $1\leq i \leq \ell$, then
the corresponding rectangles in $P$ form an ascending (respectively,
descending) F-block.
Indeed, let $H$ be such an ascending block in the graph of $R(P)$.
Let us partition the points of $H$ in singletons (formed of points
that lie on the diagonal) and pairs (formed of  transposed points
at adjacent positions).
Let $Q_1, Q_2, \dots$ be the
 parts of this partition,
read from the SW to the NE corner of $H$.
For each $i=1, 2, \dots$, the point(s) of $Q_{i+1}$
are the only NE-neighbors of the point(s) of $Q_{i}$, and, conversely,
the point(s) of $Q_{i}$ are the only SW-neighbors
of the point(s) of $Q_{i+1}$.
Therefore, by the remark that follows Observation~\ref{the:nei_dia},
the left side of the rectangle(s) corresponding to the point(s) of $Q_{i+1}$
coincides with the right side of the rectangle(s) corresponding to the point(s) of $Q_{i}$.
If $Q_i$ consists of two points then we have two rectangles whose union is a rectangle
split by a horizontal segment.
The argument is similar for a descending block.

Therefore, such
blocks in the graph of $\rho$ will be also called ascending (respectively,
descending) F-blocks.
Fig.~\ref{fig:fblocks} shows a floorplan with maximal F-blocks denoted by bold lines, and the F-blocks in the corresponding permutation $R(P)$ (the graph of $S(P)$ is also shown).

\begin{figure}[h]
$$\resizebox{120mm}{!}{\includegraphics{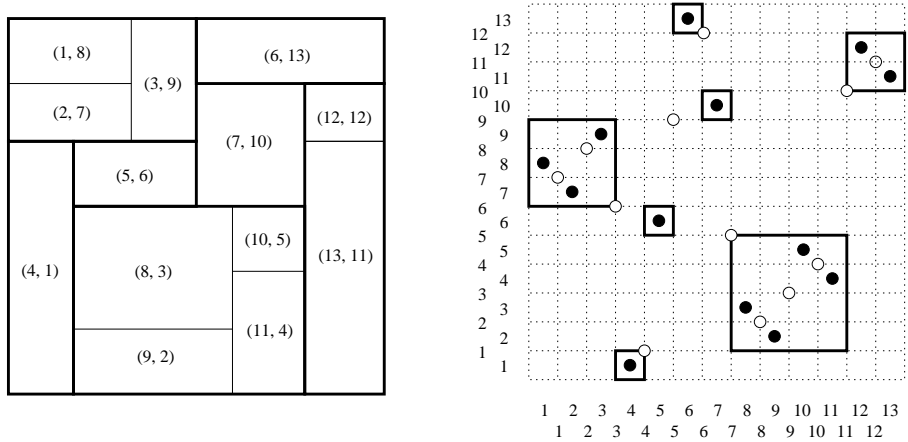}}$$
\caption{Maximal F-blocks
in floorplans and in \ps.}
\label{fig:fblocks}
\end{figure}

Let $F_1, F_2, \dots$ be all the maximal F-blocks in the graph of
$\rho$ (ordered from
left to right).
For $i\ge 1$, let $[y_i, y'_i]$ be the interval of values $\rho(j)$ occurring in
 $F_i$, and define
$d_i:=+$ if $F_i$ is ascending, and $d_i:=-$ if $F_i$ is descending
($d_i$ is left undefined if $F_i$ has size 1 or 2).
 The \emph{F-structure}
of $\rho$ is the sequence $\hat{F}_1, \hat{F}_2, \dots$,
where $\hat{F}_i = ([y_i, y'_i], d_i)$.
For example, the F-structure of the permutation
in Fig.~\ref{fig:fblocks} is
\[
\left( [7,9], +\right),  \ \ \left([1]\right), \ \ \left([6]\right),
\ \  \left([13]\right),  \ \ \left([10]\right),  \ \
\left([2,5],+\right) ,\ \  \left([11,12]\right).
\]

\begin{Theorem}
\label{the:fib}
Let $P_1$ and $P_2$ be two floorplans with $n$ segments.
Then $S(P_1) = S(P_2)$
if and only if $R(P_1)$ and $R(P_2)$ have the same F-structure.
\end{Theorem}
In other words, $S(P_1) = S(P_2)$ if and only if
$R(P_1)$ and $R(P_2)$ may be obtained from each other
by replacing some F-blocks $F_1, F_2, \dots$
with, respectively, F-blocks $F'_1, F'_2, \dots$,
where $F_i$ is S-equivalent to $F'_i$ for all $i$.

\begin{proof}
The ``if''  direction is easy to prove.
Assume  $R(P_1)$ and $R(P_2)$ have the same F-structure.
In view of the way one obtains $S(P)$ from $R(P)$ (Theorem~\ref{the:r2s}),
we have $S(P_1) = S(P_2)$. Observe in particular that inside a
maximal F-block of $R(P)$, the points of $S(P)$ are organized on the
diagonal (in the ascending case) or the anti-diagonal (in the
descending case).

In order to prove the ``only if'' direction, we will first relate, for
a point of $S(P)$, the fact of being inside a maximal F-block to the
property of being weak.
(Recall that
a point $N_i$ in the graph of $S(P)$ is weak if it has at most one
neighbor in each of the directions NW, NE, SE, SW, and strong otherwise.)
 If a maximal F-block of  $R(P)$  occupies the area $[x, x'] \times
[y, y']$, then the point $N_i=(i,j)$ is inside this block if
$x \leq i < x'$ and $y \leq j < y'$.
For example, in Fig.~\ref{fig:fblocks} six points in the graph of $S(P)$
(the white points in the combined diagram)
are inside a maximal F-block: $(1,7)$, $(2,8)$, $(8,2)$, $(9,3)$, $(10,4)$ and $(12, 11)$.
Observe that the notion of ``being inside'' a maximal F-block is
\emph{a priori} relative to
$R(P)$. However, the following proposition shows that it is an
intrinsic notion, depending on $S(P)$ only.

\begin{Lemma}\label{the:weak}
Let $N_i$ be a point in the graph of $\sigma = S(P)$.
Then $N_i$ is inside a maximal F-block of $R(P)$ if and only if it is a weak
point of $S(P)$.
\end{Lemma}

This lemma is proved in Appendix~\ref{app:weak}, and the rest of the
 theorem  in Appendix~\ref{app:fib}.
\end{proof}

\section{Enumeration of  $(2\mn14\mn3, 3\mn41\mn2)$-avoiding  permutations}
\label{sec:enum}
%

It follows from  Theorem~\ref{the:fib} that S-permutations of size $n$ are in bijection with
Baxter permutations of size $n+1$ in
which all maximal ascending F-blocks are increasing (that is, order isomorphic to a permutation of the form $123\ldots m$), and all
maximal descending F-blocks of size at least 3 are decreasing.
A Baxter \p\ that does not satisfy these conditions has at least one \emm improper pair,.
\begin{Definition}
  Let $\rho$ be a Baxter permutation. Two points of the diagram of $\rho$ that
  lie in adjacent rows and columns form an \emm improper pair, if
  they form a descent in a maximal ascending F-block, or an ascent in a maximal
  descending F-block of size at least 3.
\end{Definition}
This definition is illustrated in Fig.~\ref{fig:improper}. Observe that a point belongs to at most one improper pair. In
particular, a \p\ of size $n+1$ has at most $\lfloor \frac
{n+1}2\rfloor$ improper pairs.

\begin{figure}[h]
$$\includegraphics[width=90mm]{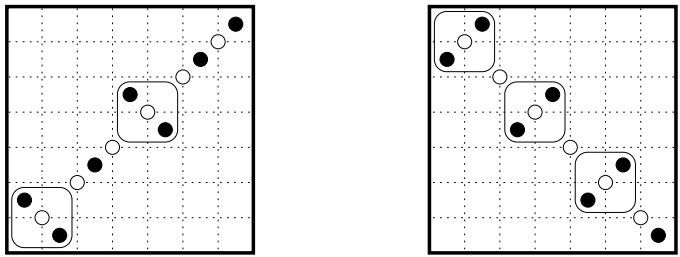}$$
  \caption{Improper pairs in maximal F-blocks.}\label{fig:improper}
\end{figure}

\begin{Proposition}\label{prop:enum}
  Let
$$
b_n= \sum_{m=0}^n\frac 2{n(n+1)^2}   {{n+1} \choose m}{{n+1} \choose {m+1}}
{{n+1} \choose {m+2}}
$$
be the number of Baxter \ps\ of size $n$ (see~{\rm\cite{chung}}).
The number $a_n$ of $(2\mn14\mn3, 3\mn41\mn2)$-avoiding permutations of size $n$ is
$$
a_n=\sum_{i=0}^{\lfloor(n+1)/2\rfloor} (-1)^i \binom{n+1-i}{i} b_{n+1-i}.
$$
\end{Proposition}
\begin{proof}
We have just explained that $(2\mn14\mn3, 3\mn41\mn2)$-avoiding
permutations of size $n$ are in bijection with Baxter \ps\ of size
$n+1$ having no improper pair. By the inclusion-exclusion principle,
$$
a_n= \sum_{i \ge 0} (-1)^ib_{n+1,i},
$$
where $b_{n+1,i}$ is the number of Baxter \ps\ of size $n+1$ with
$i$ marked improper pairs. Let $\rho$ be such a \p, and
contract every marked improper pair into a single (marked) point: this gives a Baxter
\p\ $\rho'$ of size $n+1-i$, with $i$ marked points.

Observe that if two points of $\rho$ are in the same maximal F-block,
then their images, after contraction, are in the same maximal F-block
of $\rho'$.  (The converse is false: the permutation 1342 has two
maximal F-blocks and one improper pair (consisting of the values
34). By contracting this pair, one obtains the permutation 132, which
is an F-block.)

We claim that each Baxter \p\ of size $n+1-i$ with $i$ marked points
is obtained exactly once in our construction, and that the unique way
to expand each marked  point into an improper pair is the following:
\begin{itemize}
\item if the marked point lies on the diagonal of an ascending maximal
  F-block (of size $\ge 1$),
  replace it by a descending pair of adjacent  points,
\item if the marked point lies on the anti-diagonal of a descending maximal
  F-block of size $\ge 2$,
  replace it by an ascending pair of adjacent  points,
\item otherwise, observe that the block has size at least 3; if it is ascending
  (resp.~descending), and the marked point does not lie on the
  diagonal (resp.~anti-diagonal), replace it by an ascending
  (resp. descending)  pair of adjacent
  points.
\end{itemize}
Details are left to the reader.

This construction implies that  the number of Baxter \ps\ of size $n+1$ having
$i$ marked  improper pairs is $b_{n+1,i}=\binom{n+1-i}{i} b_{n+1-i}$,
and the proposition follows.
\end{proof}

\noindent{\bf Remarks}\\
1. Let $A(t)$ be the \gf\ of  $(2\mn14\mn3, 3\mn41\mn2)$-avoiding \ps,
and let $B(t)$ be the \gf\ of (non-empty) Baxter \ps. The above
result can be rewritten as
\begin{equation}\label{gf:sol}
A(t)= \sum_{k\ge 0} t^{k}(1-t)^{k+1} b_{k+1} =  \frac 1 t  B(t(1-t)).
\end{equation}
Observe that $t(1-t)=s$ if $t=sC(s)$, where  $C(s)=\frac{1-\sqrt{1-4s}}{2s}$ is the
\gf\ of Catalan numbers. Hence, conversely,
\begin{equation}\label{cat}
B(s)= sC(s)A(sC(s)).
\end{equation}
This suggests that another connection between S- and R-permutations,
involving Catalan numbers,  exists.

\smallskip
\nid
2. The form of $a_n$ and $b_n$ implies that $A(t)$ and $B(t)$ are \emm
D-finite,, that is, satisfy a
linear differential equation with
polynomial coefficients~\cite{lipshitz-diag,lipshitz-df}. In fact,
$$
-12 t
+6(1-2 t) B(t)
-2 t \left( -3+14 t+8 {t}^{2} \right)B'(t)
-{t}^{2} \left(
  t+1 \right)  \left( 8 t-1 \right) B''(t)=0
$$
and
\begin{multline*}
12  \left( t-1 \right)  \left( 2 t-1 \right) ^{3}
+(104 t-338 {t}^{
2}+512 {t}^{3}-294 {t}^{4}-110 {t}^{5}+192 {t}^{6}-48 {t}^{7}-12)A(t)\\
-2 t \left( t-1 \right)  \left( 40 {t}^{6}-128 {t}^{5}+89 {t}^{4}+
53 {t}^{3}-88 {t}^{2}+35 t-4 \right) A'(t)
-{t}^{2} \left( 2 t-1
 \right)  \left( 8 {t}^{2}-8 t+1 \right)  \left( {t}^{2}-t-1
 \right)  \left( t-1 \right) ^{2}A''(t)=0.
\end{multline*}
This implies
that the asymptotic behavior of the numbers $a_n$ and $b_n$ can be
determined almost automatically (see for instance~\cite[Sec.~VII.9]{fs}).
For Baxter \ps, it is known~\cite{shen_chu} that $b_n \sim 8^n n^{-4}$ (up to a
multiplicative constant, which can be determined thanks to standard
techniques for the asymptotics of sums~\cite{odlyzko-handbook}). For
$a_n$, we find $    a_n \sim  (4+2\sqrt 2)^n\, n^{-4}$.

\smallskip
\nid
3. For $1\le n\le 30$, the number of $(2\mn 14 \mn 3,
3\mn41\mn2)$-avoiding permutations of $[n]$ is given in the following
table,
which we have sent to the OEIS~\cite[A214358]{oeis}.
\begin{center} \begin{tabular}{c||c||c||c||c||c}
$1$     &$374$   &$929480$   &$4023875702$   &$23320440656376$   &$161762725797343554$   \\
$2$     &$1668$  &$4803018$  &$22346542912$  &$135126739754922$  &$963907399885885724$   \\
$6$     &$7744$  &$25274088$ &$125368768090$ &$788061492048436$  &$5769548815574513550$  \\
$22$    &$37182$ &$135132886$&$709852110576$ &$4623591001082002$ &$34679563373252224012$ \\
$88$    &$183666$&$732779504$&$4053103780006$&$27277772831911348$&$209275178482957838142$
\end{tabular}
\end{center}

\section{The case of guillotine floorplans}
\label{sec:guil}

In this section we study the restriction of the map $S$ to  an
important family of floorplan called \emm guillotine
floorplans,~\cite{cardei,  gz, stockmeyer}.

\begin{Definition}\label{def:guillotine}
A floorplan $P$ is a \emph{guillotine floorplan}
(also called~\emph{slicing floorplan}~\cite{l})
if either it consists of just one rectangle, or there is a segment in $P$ that extends from one side of the boundary
to the opposite side, and
splits $P$ into two
sub-floorplans that are also guillotine.
\end{Definition}

The restriction of the map $R$ to guillotine floorplans induces a
bijection between R-equivalence classes of guillotine floorplans and
\emm separable, permutations
(defined below)~\cite{abp}.
Here, we first characterize permutations that are
 obtained as S-permutations of guillotine floorplans, and then
 enumerate them.

\subsection{Guillotine floorplans and separable-by-point permutations}

A nonempty permutation $\sigma$ is \emph{separable}
 if  it has size 1, or its graph  can be
split into two nonempty blocks $H_1$ and $H_2$, which are
themselves separable.
Then,
either
all the points in $H_1$ are to the SW of all the points of $H_2$ (then
$\sigma$, as a separable permutation, has an \emph{ascending structure}),
or
all the points in $H_1$ are to the NW of all the points of $H_2$ (then
$\sigma$, as a separable permutation, has a \emph{descending structure}).
Separable permutations are known to
coincide with $(2\mn 4\mn 1\mn 3,
3\mn 1\mn 4\mn 2)$-avoiding permutations~\cite{bbl}.
In particular, they
form a subclass of Baxter \ps.
The number $g_n$ of separable permutations of $[n]$ is the
$(n-1)$st
\emph{Schr\"{o}der number}~\cite[A006318]{oeis},
and the associated \gf\ is:
\begin{equation}\label{schroeder}
G(t):=\sum_{n\ge 1} g_n t^n = \frac {1-t-\sqrt{1-6t+t^2}}{2}.
\end{equation}

 \begin{Definition}
  A permutation $\sigma$ of $[n]$ is \emph{separable-by-point} if
it is empty, or its graph  can be split into three
 blocks $H_1$, $H_2$, $H_3$ such that
\begin{itemize}
  \item[--] $H_2$ consists of one point $N$,
  \item[--] $H_1$ and $H_3$ are themselves separable-by-point (thus,
    they may be empty),
and
  \item[--] {either} all the points of $H_1$ are to the SW of $N$, and
    all the points of $H_3$ are to the NE of $N$ (then $\sigma$ has
    an  \emph{ascending structure}),
{or} all the points of
$H_1$ are to the NW of $N$ and all the points of
$H_3$ are to the SE of $N$ (then $\sigma$ has a \emph{descending
  structure}).
\end{itemize}
 \end{Definition}

The letter $N$ for the central block refers to the fact that we have
denoted by $N_i$ the point $(i, \sigma(i))$ of an
S-permutation
$\sigma$.
Observe also that $N$ necessarily corresponds to a
fixed point of $\sigma$ if $\sigma$ is ascending,
and to a point such that $\sigma(i)=n+1-i$ is $\sigma$ is descending
and has size $n$.

\begin{figure}[ht]
\begin{center}
\resizebox{80mm}{!}{\includegraphics{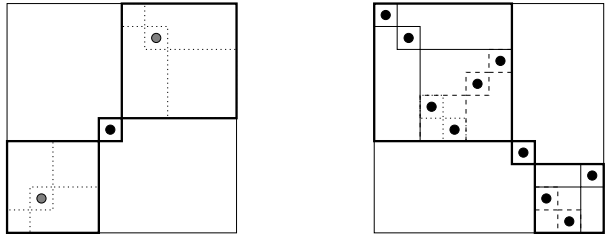}}
\caption{Separable-by-point permutations.}
\label{fig:seppoint}
\end{center}
\end{figure}

\vspace{2mm} See Fig.~\ref{fig:seppoint} for a schematic description
and an example of separable-by-point permutations.
For $n \leq 3$, all permutations are separable-by-point.
It is clear that if a nonempty permutation $\sigma$ is
separable-by-point, then it is separable.
The permutations $2143$ and $3412$ are separable, but not
separable-by-point.
The following result characterizes separable-by-point
permutations in terms of forbidden patterns.
In particular, it implies that these permutations are S-permutations.

\begin{Proposition}
 \label{the:sep_point_av}
Let $\sigma$ be a permutation of $[n]$.
Then $\sigma$ is separable-by-point if and only if it is \\
$(2 \mn 14 \mn 3, \ 3 \mn 41 \mn 2, \ 2 \mn 4 \mn 1 \mn 3, \ 3 \mn 1 \mn 4 \mn 2)$-avoiding.
\end{Proposition}

\begin{proof}
Assume that $\sigma$ is  separable-by-point.
In particular, $\sigma$ is separable, and, therefore, it avoids
$2 \mn 4 \mn 1 \mn 3$ and $ 3 \mn 1 \mn 4 \mn 2$.
Assume for the sake of contradiction that $\sigma$ contains an
occurrence of $2 \mn 14 \mn 3$,
corresponding to the points $N_i, N_j, N_{j+1}$ and $N_k$,
and has a minimal size for this property.
Then the points forming the pattern must be spread
in at least two of the three blocks.
This forces $\sigma$ to have an ascending structure, with $N_i$ and $N_j$
in one block, $N_{j+1}$ and $N_k$ in the following one
(because $N_j$ and $N_{j+1}$ are adjacent).
But this is impossible as the central block of $\sigma$ contains a unique point.
Similarly one shows that $\sigma$ avoids $3 \mn 41 \mn 2$.

Conversely, we argue by induction on the size of $\sigma$.
Let $\sigma$ be a $(2 \mn 14 \mn 3, \ 3 \mn 41 \mn 2, \ 2 \mn 4
\mn 1 \mn 3, \ 3 \mn 1 \mn 4 \mn 2)$-avoiding permutation of $[n]$.
For $n \leq 3$ there is nothing to prove. Let $n \geq 4$.
Since $\sigma$ is $(2 \mn 4 \mn 1 \mn 3, \ 3 \mn 1 \mn 4 \mn
2)$-avoiding, it is separable. Assume without loss of generality that
$\sigma$ (as a separable permutation) has an ascending structure:
the first block is $[1,i]\times[1,i]$, the second block is
$[i+1,n]\times[i+1,n]$ where $1 \leq i < n $.
If $\sigma(i) \neq i$ and $\sigma(i+1) \neq i+1$, then
$ \sigma^{-1}(i), i, i+1, \sigma^{-1}(i+1)$ form a forbidden pattern $2 \mn
14 \mn 3$. Thus, $\sigma(i) = i$ or $\sigma(i+1) = i+1$,
and one obtains a three-block decomposition of $\sigma$ by choosing for
the central block $N$ one of these two fixed points.
The remaining two blocks  avoid all four
patterns, and, therefore are
separable-by-point themselves by the induction hypothesis. It follows
that $\sigma$ is separable-by-point.
\end{proof}

\begin{Theorem}
\label{the:sep_point}
A floorplan $P$ is a guillotine floorplan if and only if $S(P)$ is
separable-by-point.
 \end{Theorem}

\begin{proof}
 Let $P$ be a guillotine floorplan.
We argue by induction on the size of $P$. If $P$ consists of  a single
rectangle, then $S(P)$ is the empty \p, and is separable-by-point.
Otherwise, consider a segment that splits $P$ into two rectangles.
  Assume that this segment is $I_i$ (that is, the $i$th segment in the
  $\tnw$~order) and that it is vertical.
All the segments to the left
 (respectively, right) of $I_i$
 come before (respectively, after) $I_i$ in  the $\tnw$ and $\tsw$ orders.
 Consequently:
 \begin{itemize}
 \item [--]  $I_i$ is also the $i$th segment in the $\tsw$~order, so
   that $N_i = (i,i)$,
\item [--]
by Observation~\ref{the:nei_dia},
all the points of the graph of $\sigma$
that correspond to  segments located to the left (respectively, right) of
$I_i$ are to the SW (respectively, NE) of $N_i$.
 \end{itemize}
 Thus, we have three blocks  $H_1$, $H_2$  and $H_3$with an ascending structure.
The blocks $H_1$ and $H_3$ are the
S-permutations of the two parts of $P$, which are themselves
guillotine:  by the
induction hypothesis, $H_1$ and $H_3$ are separable-by-point.  Thus  $S(P)$ is separable-by-point with an ascending structure.

 Similarly, if $I_i$ is horizontal, we obtain a separable-by-point
 permutation with a descending structure.

\medskip

Conversely,
assume that
$\sigma:=S(P)$ is separable-by-point.
We will prove by induction on $n$ that $P$ is a guillotine floorplan.

  The claim is clear for $n=1$. For $n > 1$,
  assume without loss of generality that $\sigma$ has an ascending
  structure. Let $H_2=\{(i,i)\}$ be the second block in a
  decomposition of $\sigma$.
Then for all $j<i$, we have $I_j \ew I_i$, and for all $j>i$, we have
$I_i \ew I_j$. Therefore, if $I_i$ is vertical, it has no below- or
above-neighbors, and thus
extends from the lower to the upper
side of the boundary.
 The two sub-floorplans of $P$ correspond respectively to the blocks
 $H_1$ and $H_3$: hence they are guillotine by the induction
 hypothesis.
 Suppose now that $I_i$ is
horizontal. Then we have $\sigma(i-1) = i-1$ (if $i>1$) and $\sigma(i+1) =
i+1$ (if $i<n$), since otherwise $I_i$ has 
several left-neighbors
or
several right-neighbors (Observation~\ref{the:nei_dia}),
which is never the case for a horizontal segment.
 Assume without loss of generality
that $i>1$. Then
another block decomposition of $\sigma$ is obtained with the central
block $H'_2=\{(i-1,i-1)\}$, corresponding to the vertical segment
$I_{i-1}$.
The previous argument then shows that $P$ is guillotine.
\end{proof}

\subsection{Enumeration}
\label{sec:multi}
In this section we enumerate
S-equivalence classes of guillotine floorplans, or equivalently,
separable-by-point permutations.
\begin{Proposition}\label{prop:enum-guillotine}
For $n\ge 1$,   let $g_n$ be the number of separable  \ps\ of size
$n$, and let $G(t)$ the associated \gf, given by~\eqref{schroeder}.
The number $h_n$ of separable-by-point permutations of size $n$ is
$$
h_n=\sum_{i=0}^{\lfloor(n+1)/2\rfloor} (-1)^i \binom{n+1-i}{i} g_{n+1-i}.
$$
Equivalently, the \gf\ of separable-by-point permutations is
$$
H(t)=\sum_{n\ge 0} h_nt^n= \sum_{n\ge 0} t^{n}(1-t)^{n+1} g_{n+1} =
\frac 1 t  G(t(1-t))=\frac{1-t+t^2-\sqrt{1-6t+7t^2-2t^3+t^4}}{2t}.
$$
\end{Proposition}
\begin{proof}
Recall that the R-permutations associated with guillotine floorplans
are the separable \ps, and return to the proof of Proposition~\ref{prop:enum}. The
contraction/expansion of points used in this proof preserves
separability, so that we can apply the same argument, which yields
directly the proposition.
\end{proof}

\nid{\bf Remarks}\\
1. The first values
 are $1, 2, 6, 20, 70, 254, 948, 3618, 14058, 55432$.
This sequence~\cite[A078482]{oeis} also enumerates $(2 \mn 4 \mn 3 \mn 1, 3 \mn 2 \mn 4 \mn 1,
2 \mn 4 \mn 1 \mn 3, 3 \mn 1 \mn 4 \mn 2)$-avoiding permutations (or \ps\ sortable by a stack of queues),
as found by Atkinson and Stitt~\cite[Thm.~17]{as}.\\
2.
Using the
\emm transfer theorems,\ from~\cite[Sec.~VI.$4$]{fs},
we can find the
asymptotic behavior of the numbers $h_n$:
\[
   h_{ n} \sim  \left(\frac{2}{1-\sqrt{8\sqrt{2}-11}}\right)^n n^{-3/2},
\]
up to a multiplicative constant.\\
3. A generalization of Proposition~\ref{prop:enum-guillotine} to
$d$-dimensional guillotine
partitions is presented in~\cite{ourselves-g}.

\section{Final remarks}
\label{sec:summary}

We have shown that many analogies exist between R- and
S-equivalence. However, there also seems to be 
 one important
difference. Looking at Fig.~\ref{fig:r_equivalence} suggests that one can transform a floorplan into an R-equivalent
one by some continuous deformation. In other words, R-equivalence
classes appear as geometric planar objects. This is confirmed by
the papers~\cite{mbm,felsner,fusy-bipolar}, which show that \emph{bipolar orientations of planar maps}
 provide a
convenient geometric description of R-equivalence classes of floorplans.
However, S-equivalence is a coarser relation, and
two S-equivalent floorplans may look rather different
(Figs.~\ref{fig:s_equivalence} and~\ref{fig:fib0}). It would be interesting to find a class of
geometric objects that captures the notion of S-equivalence
classes, as bipolar orientations do for R-equivalence classes.

In Section~\ref{sec:enum} we have established a simple enumerative
connection, involving Catalan numbers,
between Baxter \ps\ and (2-14-3,3,41-2)-avoiding \ps. Is there a
direct combinatorial proof of~\eqref{cat}, not based on
Theorem~\ref{the:fib} (the proof of which is rather heavy)?
Recall that $C(s)$ is related to pattern avoiding \ps, since
it counts $\tau$-avoiding \ps, for any pattern $\tau$ of size 3.

We conclude with a summary of the enumerative
results obtained in~\cite{abp} for R-equivalence classes
and in the present paper for S-equivalence classes.

\bigskip
  \begin{center}
\begin{tabular}{|l||l|l||}
  \hline
   & All floorplans & Guillotine floorplans \\ \hline\hline
  $\begin{array}{c}
   \textrm{R-equivalence}  \\
    \textrm{classes}
  \end{array}$&
  $\begin{array}{l}
    \textbf{Forbidden patterns: } \\
    \ 2\mn 41\mn 3, \ \, \ 3\mn 14\mn 2 \\
    \includegraphics[height=12mm]{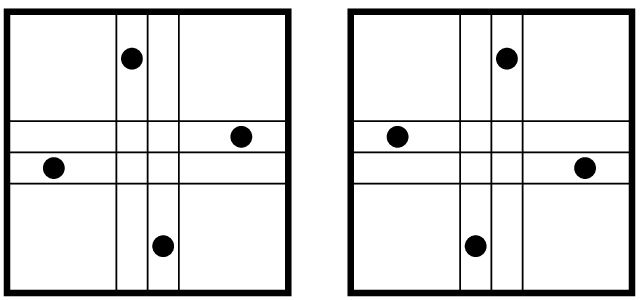} \\
    \textbf{Enumerating sequence:} \\
    1, 2, 6, 22, 92, 422, 2074, 10754,
\dots \\
    \smallskip
    \textrm{(Baxter numbers \cite[A001181]{oeis})}\\
    \smallskip
    \textbf{Growth rate: } 8 .
  \end{array}$
  &
  $\begin{array}{l}
    \textbf{Forbidden patterns: } \\
    2\mn 4 \mn 1\mn 3, \ \ 3\mn 1 \mn 4\mn 2 \\
     \includegraphics[height=12mm]{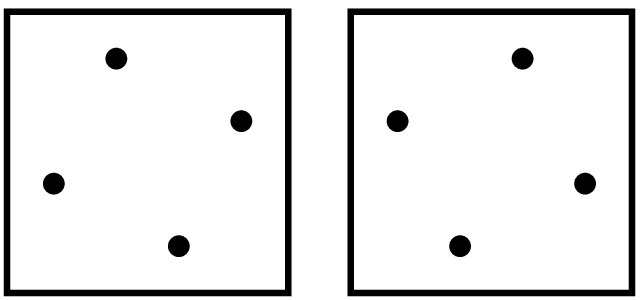} \\
    \textbf{Enumerating sequence:} \\
    1, 2, 6, 22, 90, 394, 1806, 8558,
\dots \\
    \smallskip
    \textrm{(Schr\"{o}der numbers \cite[A006318]{oeis})}\\
    \smallskip
    \textbf{Growth rate: } 3+2\sqrt{2} \approx 5.8284 .
  \end{array}$
   \\
  \hline
  $\begin{array}{c}
   \textrm{S-equivalence}  \\
    \textrm{classes}
  \end{array}$
   &
  $\begin{array}{l}
    \textbf{Forbidden patterns: } \\
    \ 2\mn 14\mn 3, \ \ \ 3\mn 41\mn 2 \\
    \includegraphics[height=12mm]{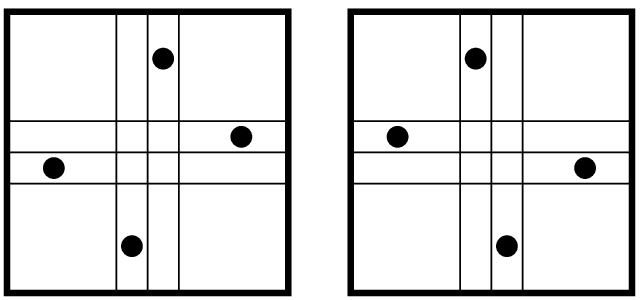} \\
    \textbf{Enumerating sequence:} \\
    1, 2, 6, 22, 88, 374, 1668, 7744,
\dots \\
    \smallskip
\textrm{(\cite[A214358]{oeis})}    \\
    \smallskip
    \textbf{Growth rate: } 4+2\sqrt{2} \approx 6.8284 .
  \end{array}$
 &
  $\begin{array}{l}
    \textbf{Forbidden patterns: } \\
    \ 2\mn 14\mn 3, \  \ 3\mn 41\mn 2, \ \, \ 2\mn 4 \mn 1\mn 3, \ \ 3\mn 1 \mn 4\mn 2 \\
    \includegraphics[height=12mm]{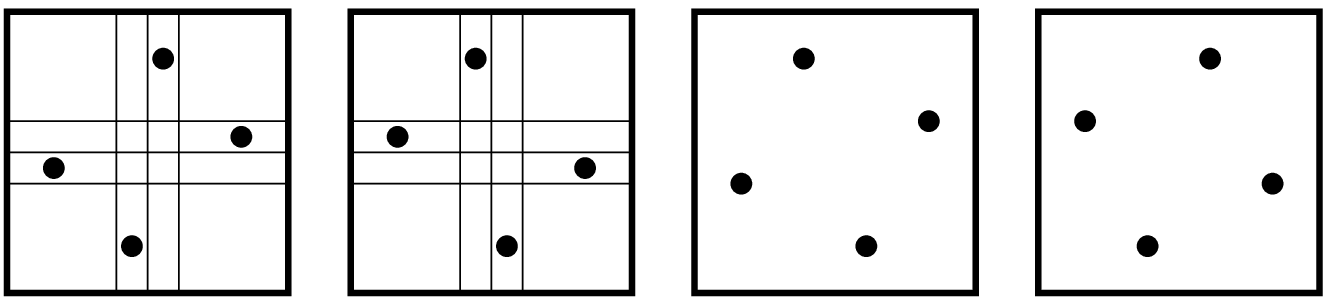} \\
    \textbf{Enumerating sequence:} \\
    1, 2, 6, 20, 70, 254, 948, 3618,
\dots \\
    \smallskip
     \textrm{(\cite[A078482]{oeis})}\\
     \smallskip
    \textbf{Growth rate: } \frac{2}{1-\sqrt{8\sqrt{2}-11}} \approx 4.5465.
  \end{array}$
 \\ \hline
  \hline
\end{tabular}
  \end{center}


\bigskip
\noindent
{\bf Acknowledgements.} We thank \'Eric
Fusy for interesting discussions on the genesis of Baxter \ps, and
also for discovering, with Nicolas Bonichon, a mistake in an earlier
(non-)proof of Proposition~\ref{prop:enum}. We also thank Mathilde
Bouvel for her help in proving that the reverse bijection $S^{-1}$ can
be implemented in linear time. Finally, we acknowledge
interesting comments from the referees of the first version of the paper.

\bigskip

\newpage
\begin{appendix}
\appendix

\section{Proof of Proposition~\ref{the:bij2},
Paragraph \boldmath{$D$}}
\label{app:sur}

\emph{We prove that, given a permutation $\sigma$,
(any) floorplan $P$ obtained by the construction
described in Paragraphs A-B
of the proof satisfies $S(P)=\sigma$.}

\nid In order to prove this claim,
we will show that for all $1 \leq i < n$, the segment $K_{i+1}$ is the
immediate successor of $K_{i}$ in the $\tnw$~order, and that
$K_{\sigma^{-1}(i+1)}$ is the immediate successor of
$K_{\sigma^{-1}(i)}$ in the $\tsw$~order.

Let us first prove that the first statement implies the second.
Let $\si'$ be obtained by applying a quarter-turn
rotation $\rho$ to $\si$ in counterclockwise direction. Let us denote by $K'_i$
the segment of $P'=\rho(P)$ containing the point $(i, \si'(i))$. By
Observation~\ref{obs:sym-reverse}, the floorplan $P'$ is
associated with $\si'$ by our construction. That is,
$\rho(K_{\si^{-1}(i)})= K'_{n+1-i}$. By assumption, $K'_{n+1-i}=\rho(K_{\si^{-1}(i)})$
follows $K'_{n-i}=\rho(K_{\si^{-1}(i+1)})$ for the $\tnw$ order in
$P'$.
Applying the quarter turn clockwise rotation $\rho^{-1}$ and
the second remark following Proposition~\ref{obs:sym-orders}, this means that
$K_{\si^{-1}(i+1)}$ follows 
$K_{\si^{-1}(i)}$ for the
$\tsw$ order in $P$.

Thus we only need to prove that  $K_{i+1}$ is the
immediate successor of $K_{i}$ in the $\tnw$~order.
By Observation~\ref{th:se_nei}, the immediate successor
of a horizontal (respectively, vertical) segment $I$ in the
$\tnw$~order is
$\R(I)$, $\LVB(I)$ or $\LHB(I)$ (respectively,
$\B(I)$, $\UHR(I)$ or $\UVR(I)$),\footnote{This notation is defined
  before Observation~\ref{th:se_nei}.} depending on
the existence of these segments
and the type of joins between them.

There are 8 cases to consider,
 depending on whether $\sigma(i)<\sigma(i+1)$ or
 $\sigma(i)>\sigma(i+1)$, and on the directions of $K_i$ and
 $K_{i+i}$.

\vspace{1mm}\nid\emph{Case $1$: $\sigma(i)<\sigma(i+1)$, $K_i$ and $K_{i+1}$ are vertical.}

Assume that $N_j$ 
bounds $K_i$ from above. Then, as shown in Paragraph C.1 above,
$K_j$ is horizontal; furthermore, $K_i$ and $K_j$ have a $\downvdash$
join at the point $(i, \sigma(j))$.
In particular, the rightmost point of $K_j$ has abscissa at least $i+1$.

If $\sigma(j) < \sigma(i+1)$, then $N_{i+1}$ bounds $K_j$ from the
right. There is a $\leftvdash$ join of $K_j$ and $K_{i+1}$ at the
point $(i+1, \sigma(j))$
(Fig.~\ref{fig:6_1-revised}(1)).

If $\sigma(j) > \sigma(i+1)$, then $N_{j}$ 
 bounds $K_{i+1}$ from above
and there is a $\downvdash$ join of $K_j$ and $K_{i+1}$
(Fig.~\ref{fig:6_1-revised}(2)).

If $N_j$ does not exist and $K_i$ reaches the upper side of the boundary,
then no point can bound $K_{i+1}$ from above, and, thus, $K_{i+1}$
reaches the boundary as well
(Fig.~\ref{fig:6_1-revised}(3)).

In all these cases, it is readily seen that $K_{i+1}$ is $\UVR(K_i)$.

By Observation~\ref{th:se_nei},
$\UVR(K_i)$ is the successor of $K_i$, unless
$\UHR(K_i)$ does not exist, $\B(K_i):=K_p$ exists and its join with
$\UVR(K_i)$ is of type $\leftvdash$
(Fig.~\ref{fig:6_1-revised}(4)).
 But this would mean that $p<i$,
and the positions of $N_i$ and $N_p$ would then contradict Observation~\ref{the:join-types}.

\begin{figure}[ht]
\begin{center}
\includegraphics[width=120mm]{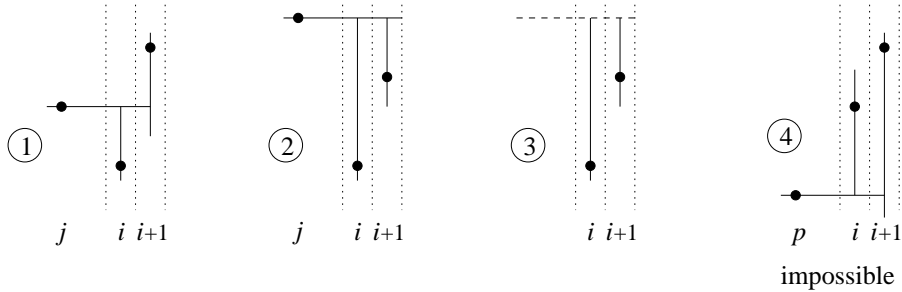}
\caption{The first case.}
\label{fig:6_1-revised}
\end{center}
\end{figure}

\vspace{1mm}\nid\emph{Case $2$: $\sigma(i)<\sigma(i+1)$, $K_i$ is vertical
  and $K_{i+1}$ is horizontal.}

The point $N_{i}$ bounds $K_{i+1}$ from the left.
Therefore, there is a $\rightvdash$ join of $K_{i}$ and $K_{i+1}$ at
the point $(i, \sigma(i+1))$, and $K_{i+1}$ is a horizontal
right-neighbor of $K_i$. Moreover, if $K_k$ is another horizontal
right-neighbor of $K_i$, then $\sigma(k) < \sigma(i+1)$: otherwise
$N_i$ cannot be a SW-neighbor of $N_k$
(Fig.~\ref{fig:6_2-revised}, left).
Therefore, $K_{i+1} = \UHR(K_i)$.

By Observation~\ref{th:se_nei},
$\UHR(K_i)$ is the successor of $K_i$, unless
$\UVR(K_i):=K_p$  exists
(Fig.~\ref{fig:6_2-revised}, right).
If this were the case, $K_p$ and $K_{i+1}$ would have a
$\upvdash$ join, and the position of $N_{i+1}$ and $N_p$ would then be
incompatible with Observation~\ref{the:join-types}.

\begin{figure}[ht]
\begin{center}
\includegraphics[width=55mm]{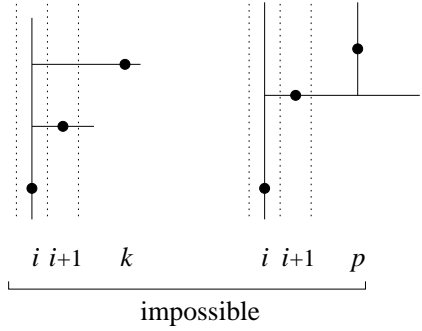}
\caption{The second case.}
\label{fig:6_2-revised}
\end{center}
\end{figure}


\vspace{1mm}\nid\emph{Case $3$: $\sigma(i)<\sigma(i+1)$, $K_i$ is horizontal
and $K_{i+1}$ is vertical.}

We claim that this  case follows from the previous one. Let $\si'$ be
obtained by applying a half-turn rotation $\rho$ to (the graph of)
$\si$. By Observation~\ref{obs:sym-reverse}, the floorplan
$P'=\rho(P)$ is associated with $\si'$. The points and segments
$\rho(N_i)$, $\rho(N_{i+1})$, $\rho(K_i)$, $\rho(K_{i+1})$ in $P'$ are
in the configuration described by Case 2, with  $\rho(N_{i+1})$ to the
left of $\rho(N_i)$. Consequently, $\rho(N_i)$ is the successor of
$\rho(N_{i+1})$ in the $\tnw$ order in $P'$. By the first remark that
follows Proposition~\ref{obs:sym-orders},
$N_{i+1}$ is the successor of $N_i$ in the $\tnw$ order in $P$.


\vspace{1mm}\nid\emph{Case $4$: $\sigma(i)<\sigma(i+1)$, $K_i$ and $K_{i+1}$ are horizontal.}

If this case, $N_{i}$ bounds $K_{i+1}$ from the left. Therefore, $K_i$
must be vertical (see Paragraph C.1 above).
Hence, this case is impossible.

\vspace{1mm}\nid\emph{Case $5$: $\sigma(i)>\sigma(i+1)$, $K_i$ and $K_{i+1}$ are vertical.}

Since $K_{i+1}$ is vertical, $N_{i+1}$ has at most one NW-neighbor, which is
then $N_i$.  By Paragraph C.1 above, $K_i$ is then horizontal. Thus this
case is impossible.

\vspace{1mm}\nid\emph{Case $6$: $\sigma(i)>\sigma(i+1)$, $K_i$ is
  vertical and $K_{i+1}$ is horizontal}

Since the segment $K_i$ is vertical, the point $N_i$ has at most one
SE-neighbor, which is then $N_{i+1}$.
Therefore, $N_{i+1}$ bounds $K_i$ from below,
and there is a $\upvdash$ join of $K_{i}$ and $K_{i+1}$ at the point
$(i, \sigma(i+1))$.
In particular, $K_{i+1}= \B(K_i)$.

By Observation~\ref{th:se_nei},
$\B(K_i)$ is the successor of $K_i$, unless
$\UHR(K_i):=K_k$  exists
(Case (3.2) in Fig.~\ref{fig:next_se}),
 or $\UHR(K_i)$ does not exist, but
$\UVR(K_i):=K_p$ does and forms with $\B(K_i)$ a $\upvdash$ join
(Case (4.1) in Fig.~\ref{fig:next_se}).
In the former case,
$K_k$ reaches $K_i$ and thus is bounded by $N_i$ on the left, but then
$N_i$ and $N_{i+1}$ are two SW-neighbors of $N_k$, and $K_k$ cannot be
horizontal.
In the latter case, $K_p$ and $K_{i+1}$ would form a $\upvdash$ join,
and the positions of $N_p$ and $N_{i+1}$ would contradict
Observation~\ref{the:join-types}.

\vspace{1mm}\nid\emph{Case $7$: $\sigma(i)>\sigma(i+1)$, $K_i$ is
  horizontal and   $K_{i+1}$ is vertical.}

This case follows from Case $6$ by the symmetry argument already  used  in Case $3$.

\vspace{1mm}\nid\emph{Case $8$: $\sigma(i)>\sigma(i+1)$, $K_i$ and $K_{i+1}$ are horizontal.}

The point that bounds $K_i$ from the right, if it exists, lies to the NE of
$N_{i+1}$.
Thus the abscissa of the rightmost point of $K_i$
is greater than or equal to
the abscissa of the rightmost point of $K_{i+1}$.

We will show that $K_{i+1}= \LHB(K_i)$. Once this is proved,
Observation~\ref{th:se_nei} implies that $\LHB(K_i)$ is the successor of $K_i$, unless
$\LVB(K_i)$ does not  exist, but $R(K_i)$ exists and   forms with
$K_{i+1}$ a $\upvdash$ join
(Case (4.2) in Fig.~\ref{fig:next_se}).
  But this would mean that $K_{i+1}$ ends
further to the right than $K_i$, which we have just proved to be
impossible.

So let us  prove that $K_{i+1}= \LHB(K_i)$. We  assume that $K_i$ does
 not reach the left side of the boundary, and that $K_{i+1}$ does not
 reach the right side of the boundary (the other cases are
proven similarly).
Let $N_{k}$ be the point that bounds $K_i$ from the left,
and let $N_{m}$ be the point that bounds $K_{i+1}$ from the right.

Consider $A$, the leftmost rectangle whose upper side is contained in $K_i$.
The left side of $A$ is clearly contained in $K_k$.
We claim that the lower side of $A$ is contained in $K_{i+1}$,
and that the right side of $A$ is contained in $K_{m}$. Note that this
implies  $K_{i+1} = \LHB(K_i)$.

Let $K_p$ (respectively, $K_q$) be the segment that contains the lower
(respectively, right) side of $A$.
Clearly,  $q>k$. If $q<i$, then $K_q$ is a vertical below-neighbor of $K_i$, and the positions of $N_q$ and $N_i$  contradict
Observation~\ref{the:join-types}. Therefore, $q > i+1$.

Consider now the segment $K_p$.
Clearly, $\sigma(p) \geq \sigma(i+1)$.
One cannot have $p>i+1$: otherwise $N_{i+1}$
(or a point located to the right of $N_{i+1}$)
would bound $K_p$ from
the left, and $K_p$ would not reach $K_k$.
One cannot have either $p<i$: otherwise $N_{i}$
(or a point located to the left of $N_{i}$)
would bound $K_p$ from
the right, and $K_p$ would not reach $K_q$. Since $p\not = i$, we have
proved that  $p=i+1$,
and $K_k$ is bounded by $N_{i+1}$ from below.

Finally, $K_{q}$ coincides with $K_m$:
otherwise, $q<m$, and $K_q$ is a vertical above-neighbor of $K_{i+1}$;
however, in this case $N_q$ would bound $K_{i+1}$ from the right,
and $K_{i+1}$ would not reach $K_m$.

We have thus proved that $K_{i+1} = \LHB(K_i)$, and this
concludes the study of this final case, and the proof of
Proposition~\ref{the:bij2}.

\section{Proof of Lemma~\ref{the:weak}}
\label{app:weak}

Let $N_i = (i,j)$ be inside a maximal F-block of $R(P)$. Assume
for the sake of contradiction that
$N_i$ is strong, and for instance,
has several
NE-neighbors. Let $N_k$ be the leftmost NE-neighbor of $N_i$, and let
$N_{\ell}$ be the lowest NE-neighbor of $N_i$.
If $k> i+1$, then
$\sigma(k-1)<\sigma(i)$, and, therefore, $i, k-1, k, \ell$ form a
forbidden pattern $2 \mn 14 \mn 3$. Thus $k= i+1$.
 Symmetrically, $\sigma(\ell) =j+1$ (Fig.~\ref{fig:weak-inside}).
Note also that $\sigma(i+1)>j+1$ and $\sigma^{-1}(j+1)>i+1$.
Since the points of $S(P)$ inside an F-block are either on the
diagonal or the anti-diagonal of this block, $N_i$ is the
highest (and rightmost) point of $S(P)$ inside the maximal F-block
that contains it, and  this F-block is of ascending type.
      In particular, either $\rho(i+1) = j+1$, or $\rho(i) = j+1$ and $\rho(i+1) = j$.

Since $\rho(i+1) \leq j+1$ and $\sigma(i+1) \geq j+2$, then $\rho(i+2)
\ge j+3$ (Theorem~\ref{the:r2s}). Symmetrically, $\rho^{-1}(j+2)\ge i+3$. But then
the position of the point $(\rho^{-1}(j+2),j+2)$
is not compatible with the position of $N_{i+1}$:
by Theorem~\ref{the:r2s}, there cannot be a point of $\rho$ located to the
right of $\rho(i+2)$ and in the rows between those of $\rho(i+1)$ and $N_{i+1}$.
Hence $N_i$ cannot have several NE-neighbors. Symmetric statements
hold for the other directions, and $N_i$ is a weak point.

\begin{figure}[ht]
  \begin{center}
  \scalebox{0.8}{\input{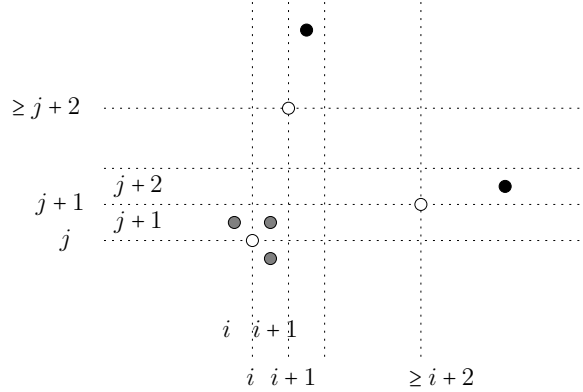}}
  \end{center}
\caption{Some points of the combined diagram of $\rho$ and
  $\sigma$. The grey points represent the two possibilities $\rho(i+1) = j+1$, or $\rho(i) = j+1$ and $\rho(i+1) = j$.}
\label{fig:weak-inside}
\end{figure}

Now let $N_i = (i, j)$ be a point of the graph of $\sigma$, not inside a
maximal F-block. Assume without loss of generality that
$\rho$ has an ascent at $i$:
  $\rho(i) \leq j < \rho(i+1)$ and (by Theorem~\ref{the:r2s})  $\rho^{-1}(j) \leq i < \rho^{-1}(j+1)$.
  We shall show that  $N_i$ has
several SW-neighbors or several NE-neighbors.
We denote $M_i=(i, \rho(i))$.

First, if $\rho(i)=j$ and $\rho(i+1)=j+1$, then $M_i$ and $M_{i+1}$
form an F-block, and $N_i$ is inside this block.
Therefore, either $\rho(i)\not=j$ or $\rho(i+1)\not=j+1$, and we may
assume without loss of generality that $\rho(i) \neq j$; hence
$\rho(i)<j$ and $\rho^{-1}(j)<i$.
Then it follows
from the definition of Baxter \ps\ that $\rho(i-1)\le j$
(otherwise, there is an occurrence of 2-41-3 at positions
$\rho^{-1}(j), i-1, i, \rho^{-1}(j+1)$).
Consequently, we have $\sigma(i-1) \leq j-1$.
Symmetrically, $\rho^{-1}(j-1) \le i $ and $\sigma^{-1}(j-1)
\leq i-1$. There are two possibilities: either $\sigma(i-1) < j-1$ and
$\sigma^{-1}(j-1) < i-1$, or $\sigma(i-1) = j-1$.
In the former case, $N_{i-1}$ and $N_{\sigma^{-1}(j-1)}$
are two SW-neighbors of $N_i$,
and we have proved that $N_i$  is strong.
Let us go on with  the latter case,
where $\rho(i-1) = j$ and $\rho(i)=j-1$.

 If $\rho(i+1)=j+1$, then $M_{i-1}$, $M_i$ and $M_{i+1}$  form an
 F-block, and $N_i$ is inside this block, which  contradicts our
 initial assumption. Otherwise,  $\rho(i+1)\not=j+1$, and an argument similar to the one
 developed just above shows that either $N_{i+1}$ and
$N_{\sigma^{-1}(j+1)}$ are two NW-neighbors of $N_i$, or $\rho(i+1)
 =j+2$ and $\rho(i+2) = j+1$. In the latter case,
$M_{i-1}$, $M_i$, $M_{i+1}$ and $M_{i+2}$ form an
 F-block containing $N_i$, which  contradicts our
 initial assumption.

 \section{Proof of Theorem~\ref{the:fib}, the ``only if'' direction}
 \label{app:fib}

 Let $\sigma$ be a $(2\mn14\mn3, 3\mn41\mn2)$-avoiding \p\ of size $n$.
Let $\cB$
be the set of Baxter \ps\ whose S-permutation (described by
Theorem~\ref{the:r2s}) is $\sigma$.
Lemma~\ref{the:weak} determines which points of the graph
 of $\sigma$ are inside an F-block.
These points
are organized along the
diagonal or anti-diagonal of their blocks.
It follows that the location of all non-trivial F-blocks
in the graph of $\rho$, for $\rho \in \mathcal{B}$,
and their type (ascending or descending, for blocks of size at least
3), are also determined uniquely. It remains to show that the location
of the trivial F-blocks (that is, F-blocks of size $1$) is also
determined 
 by $\sigma$.

Assume for the sake of contradiction that
$\mathcal{B}$ contains two distinct \ps\  $\rho_1$ and  $\rho_2  $.
Let $i$ be the abscissa of the leftmost  trivial F-block that is not
at the same ordinate in the graphs of $\rho_1$ and $\rho_2$.
 Denote $j=\sigma(i)$.
By symmetry,  we only have to consider two cases:
$(1)$ $\rho_1(i) < \rho_2(i) \leq j$; $(2)$ $\rho_1(i) \leq j < \rho_2(i)$.

In the first case (Fig.~\ref{fig:f_theorem_1}),
denote $k=\rho_2(i)$. Consider $\rho^{-1}_1(k)$.
By assumption, $\rho^{-1}_1(k) \neq i$.
Since $\rho_1(i)<k$ and $\sigma(i) = j \geq k$, we have
$\rho^{-1}_1(k)<i$ by Theorem~\ref{the:r2s}.
However, this is impossible since
the F-structures of $\rho_1$ and $\rho_2$
coincide to the left of the $i$th column.

\begin{figure}[h]
$$\includegraphics[width=110mm]{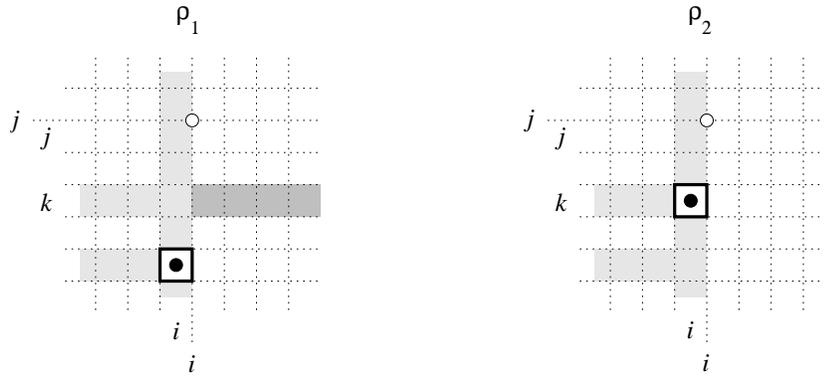}$$
 \caption{Proof of Theorem~\ref{the:fib}: the case $\rho_1(i) <
   \rho_2(i) \leq j$. The shaded areas contain no point.}
 \label{fig:f_theorem_1}
\end{figure}

Consider the second case, $\rho_1(i) \leq j < \rho_2(i)$ (Fig.~\ref{fig:f_theorem_2}).
Since $\rho_1(i) \leq j$ and $\sigma(i)=j$,
the areas $[1, i]\times\{j+1\}$ and $[i+1, n]\times\{j\}$ are empty in the graph of $\rho_1$.
Similarly, since $\rho_2(i) \geq j+1$,
the areas $[1, i]\times\{j\}$ and $[i+1, n]\times\{j+1\}$ are empty in the graph of $\rho_2$.
Since the F-structures of $\rho_1$ and $\rho_2$ coincide in
$[1, i-1] \times [1, n]$ the areas $[1, i-1]\times\{j,j+1\}$
are empty in the graphs of both permutations.
Given that rows cannot be empty, this forces $\rho_1(i) = j$ and
$\rho_2(i) = j+1$ (Fig.~\ref{fig:f_theorem_3}).

\begin{figure}[h]
$$\includegraphics[width=110mm]{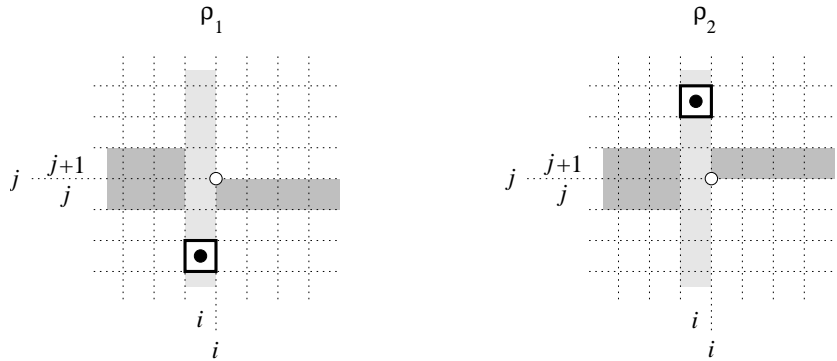}$$
  \caption{Proof of Theorem~\ref{the:fib}: the case $\rho_1(i) \leq j < \rho_2(i)$.}
  \label{fig:f_theorem_2}
\end{figure}

Assume without loss of generality that $\sigma(i+1) < j$.
Since $\rho_1(i)=j$ and $\sigma(i)=j$, we have, by Theorem~\ref{the:r2s},
$\rho_1(i+1) \geq j+1$.
In fact $\rho_1(i+1) > j+1$ since otherwise
the point $(i, \rho_1(i))$ would not form a trivial F-block.
Now, since
$\sigma(i+1) < j$, the area $[i+2, n] \times \{j+1\}$ is empty in the graph of $\rho_1$.
 This area is also empty in the graph of $\rho_2$, since
 $\rho_2(i) = j+1$.
Since the F-structures of $\rho_1$ and $\rho_2$ coincide in
 $[1, i-1] \times [1, n]$,
 the area
$[1, i-1] \times \{j+1\}$
is also empty in the graph of $\rho_1$.
Since $\rho_1(i)=j$ and $\rho_1(i+1)>j+1$,
we have a contradiction: the whole row $j+1$ is empty in the graph of $\rho_1$.

\begin{figure}[h]
$$\includegraphics[width=110mm]{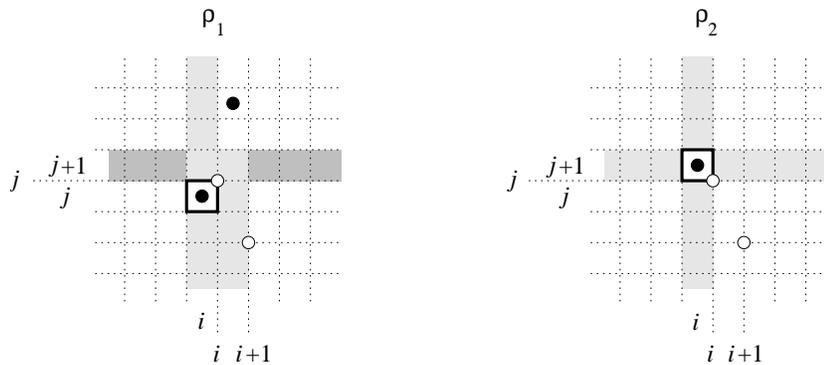}$$
  \caption{Proof of Theorem~\ref{the:fib}: the case $\rho_1(i) \leq j
    < \rho_2(i)$, continued.}
  \label{fig:f_theorem_3}
\end{figure}

Thus, we have proved that all $\rho \in \mathcal{B}$
have the same F-structure.

\medskip

\hrule

\end{appendix}

\end{document}